\documentclass{article}
\usepackage[OT2,OT1]{fontenc}
\usepackage{amsfonts,amsmath, amssymb,latexsym,mathrsfs,array,multirow}
\usepackage{graphicx}
\usepackage[all,cmtip]{xy}
\usepackage{float}
\usepackage{hyperref}
\usepackage{color}
\def\Sha{\mbox{\fontencoding{OT2}\selectfont\char88}}

\def\Z{{\mathbb Z}}
\def\A{{\mathbb A}}
\def\bbz{{\mathbb Z}}
\def\bbg{{\mathbb G}}

\newcommand{\udots}{\mathinner{\mskip1mu\raise1pt\vbox{\kern7pt\hbox{.}}\mskip2mu\raise4pt\hbox{.}\mskip2mu\raise7pt\hbox{.}\mskip1mu}}

\def\SL{{\rm SL}}
\def\GL{{\rm GL}}
\def\Rlk{{{\rm Res}_{L/k}}}
\newcommand{\lrg}[1] {\langle#1\rangle}
\def\tStab{{\rm Stab}}
\def\SO{{\rm SO}}
\def\PSO{{\rm PSO}}
\def\tSpec{{\rm Spec}}

\def\sol{{\rm sol}}
\def\Stab{{\rm Stab}}
\def\Sel{{\rm Sel}}
\def\Inv{{\rm Inv}}

\def\Gal{{\rm Gal}}

\def\P{{\mathbb P}}

\def\disc{{\rm disc}}

\def\al{{\alpha}}

\def\irr{{\rm irr}}
\def\bbp{{\mathbb P}}
\def\sco{{\mathscr O}}
\def\dim{{\rm dim}}
\def\red{{\rm red}}
\def\dcup{{\,\dot\cup}}
\def\pico{{{\rm Pic}^1}}
\def\tTr{{\rm Tr}}
\def\tSpan{{\rm Span}}
\def\be{{\beta}}

\def\Vol{{\rm Vol}}
\def\R{{\mathbb R}}
\def\F{{\mathbb F}}
\def\scj{{\mathcal J}}
\def\tEnd{{\rm End}}
\def\tdiv{{\rm div}}
\def\bbr{{\mathbb R}}
\newcommand{\w}[1]{\widetilde{#1}}
\def\bbq{{\mathbb Q}}
\def\FF{{\mathcal F}}

\def\Q{{\mathbb Q}}
\def\bbq{{\mathbb Q}}
\def\sco{{\mathcal O}}
\def\H{{\mathcal H}}
\def\J{{\mathcal J}}

\def\Z{{\mathbb Z}}

\def\P{{\mathbb P}}
\newcommand{\gl}[2]{{^{#2\!}#1}}
\def\F{{\mathbb F}}
\def\Q{{\mathbb Q}}

\def\H{{\mathcal H}}

\def\wzn2{{W_{\Z,+}^{(2-)}}}

\def\fz1{{F_{\Z,1}}}
\def\tO{{\rm O}}
\def\tSO{{\rm SO}}
\def\tPSO{{\rm PSO}}
\def\PO{{\rm PO}}

\def\SO{{\rm SO}}
\def\PSO{{\rm PSO}}
\hypersetup{colorlinks=true, linkcolor=black}

\newcommand{\qu}[2]{\langle#1,#2\rangle_Q}

\def\quu{{\qu{\;}{\hspace{1pt}}}}

\newtheorem{theorem}{Theorem}[section]
\newtheorem{corollary}[theorem]{Corollary}
\newtheorem{cor}[theorem]{Corollary}
\newtheorem{lemma}[theorem]{Lemma}

\newtheorem{proposition}[theorem]{Proposition}
\newtheorem{defn}[theorem]{Definition}
\newenvironment{proof}{\noindent {\bf Proof:}}{$\Box$ \vspace{2 ex}}

\begin{document}

\title{Average size of the $2$-Selmer group of Jacobians of monic even hyperelliptic~curves}

\author{Arul Shankar and Xiaoheng Wang}
\maketitle
\begin{abstract}
  In \cite{BG3}, Manjul Bhargava and Benedict Gross considered the
  family of hyperelliptic curves over $\Q$ having a fixed genus and a
  marked rational Weierstrass point. They showed that the average size of the
  $2$-Selmer group of the Jacobians of these curves, when ordered by
  height, is $3$. In this paper, we consider the family of
  hyperelliptic curves over $\Q$ having a fixed genus and a marked rational
  non-Weierstrass point. We show that when these curves are ordered by
  height, the average size of the $2$-Selmer group of their Jacobians
  is $6$. This yields an upper bound of $5/2$ on the average rank of
  the Mordell-Weil group of the Jacobians of these hyperelliptic
  curves.

  Finally using an equidistribution result, we modify the techniques of
  \cite{PoSto} to conclude that as $g$ tends to infinity, a proportion
  tending to 1 of these monic even-degree hyperelliptic curves having genus $g$ have
  exactly two rational points---the marked point at infinity and its hyperelliptic conjugate.
\end{abstract}
\tableofcontents

\setcounter{tocdepth}{4}


\section{Introduction}
In \cite{BG3}, Manjul Bhargava and Benedict Gross studied
hyperelliptic curves over $\bbq$ with a rational Weierstrass
point. Any such curve of genus $g$ can be given as the smooth projective model of
the affine curve defined by $$y^2 =
x^{2n+1}+c_2x^{2n-1}+\cdots+c_{2n+1},$$ where $c_i\in \bbq$ and the
given rational Weierstrass point lies above $x=\infty.$ If we further
assume that $c_i\in \bbz$ and that there is no prime $p$ such that
$p^{2i}$ divides $c_i$ for all $i$, then such an expression is
unique. The height $H$ of such a curve $C$ is defined by
$$H(C) = \max\{|c_k|^{2n(2n+1)/k}\}_{k=2}^{2n+1}.$$
Bhargava and Gross showed:

\begin{theorem}{\rm (\cite[Theorem 1.1]{BG3})}\label{thm:oddmainBG3}
  When all hyperelliptic curves of fixed genus $n\geq1$ over $\bbq$
  having a rational Weierstrass point are ordered by height, the
  average size of the 2-Selmer group of their Jacobians is 3.
\end{theorem}
As an immediate corollary, they obtained that the average rank of the
Mordell-Weil groups of the Jacobians of such curves is at most
3/2. For a concise summary of their results and the techniques used in the proofs, see \cite{GHanoi}.

In this paper, we consider hyperelliptic curves of genus $n\geq2$ over
$\bbq$ with a marked rational non-Weierstrass point that we will
denote by $\infty$. Any such curve $C$ also has a second rational
point $\infty'$, namely the conjugate of $\infty$ under the
hyperelliptic involution. In other words, $\infty'$ is the unique
point in $C(\bar{\bbq})$ such that $h^0(\sco_C(\infty+\infty'))=2.$ By
studying $H^0(C, k\cdot(\infty+\infty')),$ one can show that $C$ can
be given as the smooth projective model of the affine curve defined
by
\begin{equation}\label{eqevenhyperformula}
y^2 = x^{2n+2}+c_2x^{2n}+\cdots+c_{2n+2}
\end{equation}
where $c_i\in \bbq$
and the points $\infty,\infty'$ lie above $x=\infty.$ If we further
assume that $c_i\in \bbz$ and that there is no prime $p$ such that
$p^{2i}$ divides $c_i$ for all $i$, then such an expression is
unique. We analogously define the height $H$ of $C$ by
$$H(C) = \max\{|c_k|^{(2n+1)(2n+2)/k}\}_{k=2}^{2n+2}.$$

Recall that the 2-Selmer group $\Sel_2(J)$ of the Jacobian
$J=\text{Jac}(C)$ of $C$ is a finite subgroup of the Galois cohomology
group $H^1(\bbq,J[2]),$ which is defined by local conditions and fits
into an exact sequence
$$0\rightarrow J(\bbq)/2J(\bbq)\rightarrow\Sel_2(J)\rightarrow \Sha_2(\bbq,J)\rightarrow0,$$
where $\Sha_2(\bbq,J)$ denotes the Tate-Shafarevich group of $J$ over
$\bbq.$

The main result of this paper is:

\begin{theorem}\label{thm:evenmainssss}
  When all hyperelliptic curves of fixed genus $n\geq2$ over $\bbq$
  having a marked rational non-Weierstrass point are ordered by
  height, the average size of the 2-Selmer group of their Jacobians is
  6.
\end{theorem}

More precisely, we show that
$$\lim_{X\rightarrow\infty}\frac{\displaystyle\sum_{H(C)<X}\#\Sel_2(\text{Jac}(C))}{\displaystyle\sum_{H(C)<X}1}=6,$$
where $C$ ranges over all hyperelliptic curves of the form \eqref{eqevenhyperformula}.
In fact, we prove that the same result remains true even when we
average over any subset of hyperelliptic curves $C$ defined by a
finite set of congruence conditions on the coefficients
$c_2,c_3,\ldots,c_{2n+2}.$

We impose the condition that $n\geq2$ because every point on a genus 1
curve is a Weierstrass point.
We will show in Proposition \ref{lem:100distinguish} that the class
$(\infty')-(\infty)$ is not divisible by 2 in $J(\bbq)$ for a 100\% of
hyperelliptic curves with a marked rational non-Weierstrass point.
Therefore we expect the 2-Selmer groups of these Jacobians to have, on
average, one extra generator compared to the Jacobians of
hyperelliptic curve with one marked Weierstrass point. In other words:
given Theorem \ref{thm:oddmainBG3}, we expect Theorem \ref{thm:evenmainssss} to
be true.  Now
when $(\infty')-(\infty)$ is not divisible by 2 in $J(\bbq)$, the
average 2-rank of the 2-Selmer group minus 1 is at most 3/2. This
follows because $|\Sel_2(J)|/2$ is at least 1 and the average is 3 as
$C$ runs through hyperelliptic curves with a marked rational
non-Weierstrass point. Therefore we obtain the following result.

\begin{cor}\label{cor:evenrankbound}
  When all hyperelliptic curves of fixed genus $n\geq2$ over $\bbq$
  having a marked rational non-Weierstrass point are ordered by
  height, the average rank of the 2-Selmer group of their Jacobians is
  at most 5/2. Thus the average rank of the Mordell-Weil groups of
  their Jacobians is at most 5/2.
\end{cor}

In \cite{BG3}, Bhargava and Gross also used a method of Chabauty
\cite{Cha}, \cite{ChCo} to show that when $g\geq2$, a positive
proportion of hyperelliptic curves of genus $g$ with a rational
Weierstrass point have at most 3 rational points; and when $g\geq3,$ a
majority of such curves have at most 20 rational points. (These
hyperelliptic curves having genus $g$ correspond to the affine
equation $y^2=x^{2g+1}+\cdots+c_{2g+1}$.)  In \cite{PoSto}, Poonen and
Stoll used Chabauty's method and the results of \cite{BG3} to show that
a positive proportion of odd degree hyperelliptic curves having a
fixed genus $g\geq3$ have exactly one rational point -- the
Weierstrass point at infinity -- and that this proportion tends to 1
as $g$ tends to infinity. Analogously, we show that in our case, a
positive proportion of even degree hyperelliptic curves of genus
$g\geq 10$ have exactly two rational points -- the marked
non-Weierstrass point $\infty$ at infinity and its image $\infty'$
under the hyperelliptic involution. We also show that as $g$ tends to
infinity, this proportion tends to 1. More precisely, we prove the
following theorem:
\begin{theorem}\label{thchab}
  The proportion of monic even degree hyperelliptic curves having genus $g\geq 4$
  that have exactly two rational points is at least
  $1-(48g+120)2^{-g}$.
\end{theorem}


To prove Theorem \ref{thm:evenmainssss}, we follow the same strategy
as \cite{BS}, \cite{BS3} and \cite{BG3}. Let $(U,Q)$ denote the split
quadratic space of dimension $2n+2$ over $\bbq$ and let $V$ denote the
space of operators $T$ on $U$ self-adjoint with respect to $Q$. For
any monic separable polynomial $f(x)$ of degree $2n+2$, let $J_f$
denote the Jacobian of the hyperelliptic curve defined by the affine
equation $y^2 = f(x),$ and let $V_f$ denote the subscheme of $V$
consisting of self-adjoint operators $T$ with characteristic
polynomial $f(x)$. In Section \ref{sec:orbitparameterization}, we
obtain a bijection between $\Sel_2(J_f)$ and locally
soluble orbits of the conjugation action of $\PSO(U)$ on $V_f.$ This
parameterization step can be viewed as an example of Arithmetic
Invariant Theory. Although not strictly needed, the arithmetic theory
of pencils of quadrics as developed in \cite{Jerrythesis} can be used to
give a very nice geometric interpretation of solubility. More
precisely, a self-adjoint operator $T\in V_f(\bbq)$ is soluble if and
only if there exists a rational $n$-plane $X$ that is isotropic with
respect to the following two quadrics:
\begin{eqnarray*}
Q(v) &=& \qu{v}{v}\\
Q_T(v) &=& \qu{v}{Tv},
\end{eqnarray*}
where $\quu$ is the bilinear form associated to $Q$. A self-adjoint
operator $T\in V_f(\bbq)$ is locally soluble if and only if such an
$n$-plane exists locally everywhere.

In Section \ref{sec:orbitcounting}, we count the number of locally
soluble orbits using techniques of Bhargava developed in
\cite{dodqf}. We count first the number of integral orbits soluble at
$\bbr$ by counting the number of integral points inside a fundamental
domain for the action of $\PSO(U)(\R)$ on $V(\R)$. We break up this
fundamental domain into a compact part and a cusp region where
separate estimations are required. The compact part of the fundamental
domain will contribute to, on average, four Selmer elements. The cusp
region corresponds to the two ``obvious'' classes: $0$ and
$(\infty')-(\infty)$.  The second step is a sieve to the locally
soluble orbits by imposing infinitely many congruence conditions. For
this the uniformity estimates of \cite{geosieve} are needed.

In Section 5, we combine the results from previous sections to prove
Theorems \ref{thm:evenmainssss}. Finally in Section~6, we modify the
methods of \cite{PoSto} to prove Theorem \ref{thchab}.
\section{Orbit parameterization}
\label{sec:orbitparameterization}


Let $k$ be a field of characteristic not 2 and let $(U,Q)$ be the
(unique) split quadratic space over $k$ of dimension $2n+2$ and discriminant 1. Let
$f(x)$ be a monic polynomial of degree $2n+2$ with no repeated roots
and splitting completely over $k^s$. In this section, we study the
action of $\PSO(U)$ on self-adjoint operators of $U$
with characteristic polynomial $f(x)$ via conjugation. More precisely, let
$\qu{v}{w}=Q(v+w)-Q(v)-Q(w)$ denote the bilinear form associated to
$Q$. For any linear operator $T:U\rightarrow U$, its adjoint $T^*$ is
defined via the following equation:
$$\qu{Tv}{w}=\qu{v}{T^*w},\quad\forall\; v,w\in U.$$
Let $V$ denote the $k$-scheme
$$
V = \{T:U\rightarrow U|T = T^*\},
$$
and $V_f$ the $k$-scheme
$$V_f = \{T:U\rightarrow U|T = T^*,\det(xI - T)=f(x)\}.$$
The group scheme $$\SO(U):=\{g\in \text{GL}(U)|gg^*=I,\det(g)=1\}$$
acts on $V_f$ via $g\cdot T = gTg^{-1}.$ The center
$\mu_2\leq\SO(U)$ acts trivially. Hence we obtain a faithful action
of $$G=\PSO_{2n+2}:=\PSO(U)=\SO(U)/\mu_2.$$


To study the orbits of these actions, we first work over the separable
closure $k^s$ of $k$ in \S2.1 and show that $G(k^s)$ acts transitively
on $V_f(k^s)$ for separable polynomials $f$. In \S2.2, we work over
$k$ and classify the $G(k)$-orbits on $V_f(k)$ using Galois
cohomology. In \S2.3, we consider the Jacobian $J$ of the
hyperelliptic curve given by the equation $y^2=f(x)$ and obtain a
bijection between $G(k)\backslash V_f(k)$ and a subset of
$H^1(k,J[2])$. The most difficult part of this section will be to show
that this subset contains the image of $J(k)/2J(k)$ in
$H^1(k,J[2])$. Finally, in \S2.4, we work over $\Z_p$ and describe the
integral orbits $G(\Z_p)\backslash V(\Z_p)$.

\subsection{Geometric orbits}
\begin{proposition}\label{prop:geoorbit}
  The group $G(k^s)$ acts transitively on $V_f(k^s).$ For any $T\in
  V_f(k),$ the stabilizer subscheme $\tStab_G(T)$ is isomorphic to
  $(\text{Res}_{L/k}\mu_2)_{N=1}/\mu_2,$ where $L=k[x]/f(x)$ is an
  etale $k$-algebra of dimension $2n+2$.
\end{proposition}
\begin{proof}
  Fix any $T$ in $V_f(k).$ Since $T$ is regular semi-simple, its
  stabilizer scheme in $\GL(U)$ is a maximal torus. It contains and
  hence equals to the maximal torus $\mbox{Res}_{L/k}\bbg_m.$ For any
  $k$-algebra $K$, we have
$$\tStab_{\tO(U)}(T)(K) = \{g\in(K[T]/f(T))^\times|g^*g = 1\}.$$
Since $T=T^*$ and $g$ is a polynomial in $T$, we have $g=g^*$. Thus,
\begin{eqnarray*}
\tStab_{\tO(U)}(T) &\simeq& \tStab_{\GL(U)}(T)[2]\simeq\mbox{Res}_{L/k}\mu_2,\\
\tStab_{\tSO(U)}(T) &\simeq& (\mbox{Res}_{L/k}\mu_2)_{N=1},\\
\tStab_{\tPSO(U)}(T) &\simeq& (\mbox{Res}_{L/k}\mu_2)_{N=1}/\mu_2.
\end{eqnarray*}

Since $T$ is self-adjoint, there is an orthonormal basis
$\{u_1,\ldots,u_{2n+2}\}$ for $U$ consisting of eigenvectors of $T$
with eigenvalues $\lambda_1,\ldots,\lambda_{2n+2}$. If $T'$
is another elements of $V_f(k^s),$ then there is an orthonormal basis
$\{u'_1,\ldots,u'_{2n+2}\}$ of $U$ consisting of eigenvectors of $T'$
with eigenvalues $\lambda_1,\ldots,\lambda_{2n+2}$. Let $g\in
\SL(U)(k^s)$ be an operator sending $u_i$ to $\pm u'_i$. Then $g\in
\SO(U)(k^s)$ and the image of $g$ in $\PSO(U)(k^s)$sends $T$ to $T'$.
\end{proof}

\subsection{Rational orbits via Galois cohomology}
Our first aim is to show that $V_f(k)$ is non-empty. Indeed, one can view $L=k[x]/f(x)$ as
a $2n+2$ dimensional $k$-vector space with a power basis
$\{1,\beta,\ldots,\be^{2n+1}\}$ where $\be\in k[x]/f(x)$ is the image
of $x$. We define the binear form $<,>$ on $L$ as follows:
$$<\lambda,\mu>:=\mbox{ coefficient of }\be^{2n+1}\mbox{ in }\lambda\mu=\tTr_{L/k}(\lambda\mu/f'(\be)).$$
This form is split since the $n+1$ plane
$Y=\tSpan\{1,\be,\ldots,\be^{n}\}$ is isotropic. Its discriminant is 1 as one can readily compute using the above power basis. By the uniqueness of
split quadratic spaces of fixed dimension and discriminant 1, there exists an
isometry between $(L,<,>)$ and $(U,\quu)$, well defined up to post
composition by elements in $\tO(U)(k)$.
Let $\cdot\be:L\to L$ denote the linear map given by multiplication by $\be$.
Then $\cdot\be$ is self-adjoint with characteristic polynomial $f(x)$,
and hence yields an element in $V_f(k)$ well-defined up to $\tO(U)(k)$
conjugation. In what follows, we fix an isometry $\iota:L\to U$ thus
yielding a fixed element $T_f\in V_f(k)$.


Given $T\in V_f(k)$ there exists $g\in G(k^s)$ such that $T =
gT_fg^{-1}$, since there is a unique geometric orbit (see Proposition
\ref{prop:geoorbit}).  For any $\sigma\in\Gal(k^s/k),$ the element
$\gl{g}{\sigma}$ also conjugates $T_f$ to $T$ and hence
$g^{-1}\gl{g}{\sigma}\in \tStab_G(T_f)(k^s).$ The 1-cochain $c_T$
given by $(c_T)_\sigma=g^{-1}\gl{g}{\sigma}$ is a 1-cocycle whose
image in $H^1(k,G)$ is trivial.
This defines a bijection
\begin{eqnarray}\label{eqgtoh}
G(k)\backslash V_f(k)&\leftrightarrow& \ker(H^1(k,\tStab_G(T_f))\rightarrow H^1(k, G))\\
T&\mapsto& c_T.
\end{eqnarray}
See \cite[Proposition 1]{AIT} for more details.

\subsubsection*{Distinguished orbits}
We call a self-adjoint operator $T\in V_f(k)$ \textbf{distinguished}
if it is $\PO(U)(k)$-equivalent to $T_f$. Since the $\PO(U)(k)$-orbit
of $T_f$ might break up into two $\PSO(U)(k)$-orbits, there might
exist two distinguished $\PSO(U)(k)$-orbits in contrast to the odd
hyperelliptic case.  As $\tStab_{PO(U)}(T_f)\simeq
\mbox{Res}_{L/k}\mu_2/\mu_2$, we have the following diagram of exact
rows:
\begin{displaymath}
\xymatrix{
\bigl(\mbox{Res}_{L/k}\mu_2/\mu_2\bigr)(k) \ar[r]^{\hspace{20pt}N}& \mu_2(k)\ar[r]\ar[d]^{\sim}& H^1(k, \tStab_{{\rm PSO}(U)}(T_f))\ar[r]\ar[d] & H^1(k, \tStab_{{\rm PO}(U)}(T_f))\ar[d]\\
\PO(U)(k)\ar@{->>}[r]&\mu_2(k)\ar[r]&H^1(k, \PSO(U))\ar[r]&H^1(k, \PO(U)).}
\end{displaymath}
Therefore a self-adjoint operator $T\in V_f(k)$ is distinguished if and only if
$$c_T\in\ker(H^1(k, \tStab_{{\rm PSO}(U)}(T_f))\rightarrow H^1(k, \tStab_{{\rm PO}(U)}(T_f))).$$
Since $H^1(k, \PSO(U))\rightarrow H^1(k, \PO(U))$ is injective, every class in the above kernel corresponds to a $\PSO(U)(k)$-orbit.

Distinguished $\PSO(U)(k)$-orbits in $V_f(k)$ are unique if
and only if the norm map $N:\mbox{Res}_{L/k}\mu_2/\mu_2(k) \rightarrow
\mu_2(k)$ is surjective. Therefore, \cite[Lemma 11.2]{PoonenShaefer}
immediately implies the following result.

\begin{proposition}\label{prop:poonenover2}
  The set of distinguished elements in $V_f(k)$ consists of a single
  $\PSO(U)(k)$-orbit if and only if one of the following conditions is
  satisfied:
\begin{itemize}
\item[{\rm (1)}] $f(x)$ has a factor of odd degree in $k[x]$.
\item[{\rm (2)}] $n$ is even and $f(x)$ factors over some quadratic extension $K$
  of $k$ as $h(x)\bar{h}(x)$, where $h(x)\in K[x]$ and $\bar{h}(x)$ is
  the $\Gal(K/k)$-conjugate of $h(x).$
\end{itemize}
Otherwise, the set of distinguished elements in $V_f(k)$ consists of
two $\PSO(U)(k)$-orbits. Condition $(2)$ is equivalent to saying that $n$ is
even, and $L$ contains a quadratic extension $K$ of $k$.
\end{proposition}

To give a more explicit description of distinguished orbits, we have the following result, the proof of which is deferred to Section \ref{sec:pencil}.

\begin{proposition}\label{prop:Wt}
A self-adjoint operator $T\in V_f(k)$ is distinguished if and only if there exists a $k$-rational $n$-plane $X\subset U$ such that $\tSpan\{X,TX\}$ is an isotropic $n+1$ plane.
\end{proposition}

After a change of basis, we may take the matrix $A$ with 1's on the anti-diagonal and 0's elsewhere as a Gram matrix for $Q$. We express this basis as $$\{e_1,\ldots,e_{n+1},f_{n+1},\ldots,f_1\}$$
where
\begin{equation}\label{eq:stdbasis}
\qu{e_i}{f_j}=\delta_{ij},\quad \qu{e_i}{e_j}=0=\qu{f_i}{f_j}.
\end{equation}
We call this the standard basis. Then the above proposition yields the
following explicit description of distinguished elements which will be
useful in Section \ref{sec:orbitcounting}.

\begin{proposition}\label{prop:distexplicit}
A self-adjoint operator in $V_f(k)$ is distinguished if and only if its $\PSO(U)(k)$-orbit contains an element $T$ whose matrix $M$, with respect to the standard basis, satisfies
\begin{equation}\label{eq:distinguished}
AM=
\begin{pmatrix}
0&0&\cdots&0&0&*&*\\
0&0&\cdots&0&*&*&*\\
\vdots&\vdots&\udots&\udots&\vdots&\vdots&\vdots\\
0&0&\udots&\cdots&\vdots&\vdots&\vdots\\
0&*&\cdots&\cdots&*&*&*\\
*&*&\cdots&\cdots&*&*&*\\
*&*&\cdots&\cdots&*&*&*
\end{pmatrix}.
\end{equation}
\end{proposition}
\begin{proof}
  The forward direction follows from an argument identical to the
  proof of \cite[Proposition 4.4]{BG3}. For the backwards direction,
  suppose $AM$ has the form in \eqref{eq:distinguished}. Then
\begin{equation}\label{eq:Tei}
Te_i\in\tSpan\{e_1,\ldots,e_{n+1}\}^\perp=\tSpan\{e_1,\ldots,e_{n+1}\},\quad\mbox{for }i=1,\ldots,n.
\end{equation}
Let $X$ be the $n$-plane $\tSpan\{e_1,\ldots,e_n\}.$ Since $T$ is
self-adjoint, its eigenspaces are pairwise orthogonal. Since $Q$ is
non-degenerate, none of the eigenvectors of $T$ is isotropic. As a result,
no isotropic linear space is $T$-stable. Therefore by
\eqref{eq:Tei}, $$\tSpan\{X,TX\}=\tSpan\{e_1,\ldots,e_{n+1}\}.$$ By
Proposition \ref{prop:Wt}, $T$ is distinguished.
\end{proof}

\subsubsection*{Remaining orbits}
We start by describing the set of $\tO(U)(k)$-orbits on $V_f(k)$.
Recall that $\tStab_{{\rm O}(U)}(T_f)\simeq\Rlk\mu_2$. The set
$$\ker(H^1(k,\tStab_{{\rm O}(U)}(T_f))\rightarrow H^1(k,O(U)))$$
consists of elements $\al\in H^1(k,\Rlk\mu_2)\simeq L^\times/L^{\times2}$ whose image in
$H^1(k,O(U))$ is trivial. For any $\al\in L^\times/L^{\times2},$ lift
it arbitrarily to $L^\times$ and consider the following bilinear form
on $L$:
$$<\lambda,\mu>_\al = \mbox{ coefficient of }\be^{2n+1}\mbox{ in }\al\lambda\mu = \tTr_{L/k}(\al\lambda\mu/f'(\be)).$$
We claim that $\al$ maps to $0$ in $H^1(k, \tO(U))$ if and only if
$<,>_\al$ is split with discriminant 1. Indeed, let $\iota:(L,<,>)\rightarrow (U,\quu)$ denote the
isometry used to define $T_f$.
Now $<,>_\al$ is split with discriminant 1 if and only if there exists $g\in O(U)(k^s)$ such that the following composite map is defined over $k$:
\begin{equation}\label{eq:isoG}
  (L,<,>_\al)\xrightarrow{\sqrt{\al}}_{k^s} (L,<,>)\xrightarrow{\iota}_k (U,\quu)\xrightarrow{g}_{k^s} (U,\quu),
\end{equation}
where the subscripts below the arrows indicate the fields of
definition and where the last map is the standard action of
$g\in\text{O}(U)(k^s).$ Unwinding the definitions (\cite[Proposition
2.13]{Jerrythesis}), we see that this is equivalent to the image of
$\al$ mapping to $0$ in $H^1(k, \tO(U))$. We have therefore shown
the following result.

\begin{theorem}\label{thm:Oorbits}
  There is a bijection between $\text{O}(U)(k)$-orbits on $V_f(k)$
  and classes $\al\in(L^\times/L^{\times2})_{N=1}$ such that $<,>_\al$
  is split.
\end{theorem}

To study $\SO(U)(k)$- and $\PO(U)(k)$-orbits, we note that all the maps in the following diagram are injections.
\begin{displaymath}
\xymatrix{
H^1(k, \SO(U))\ar[r]\ar[d]&H^1(k, \text{O}(U))\ar[d]\\
H^1(k, \PSO(U))\ar[r]&H^1(k, \PO(U))}
\end{displaymath}
The horizontal maps are injective because
$\det:{\rm{O}}(U)(k)\rightarrow\mu_2(k)$ is surjective. The vertical maps
are injective because the connecting homomorphism
$\PSO(U)(k)\rightarrow k^\times/k^{\times2}$ is surjective. Indeed,
for any $c\in k^\times,$ the element in $\PSO(U)(k)$ mapping to $c$
is the operator
$$e_i\mapsto\sqrt{c}\,e_i,\quad f_i\mapsto\sqrt{c}f_i,\quad\forall i=1,\ldots,n+1.$$
Recall that $\Stab_\SO(T_f)\simeq (\Rlk\mu_2)_{N=1}$. From the exact sequence $$1\rightarrow (\Rlk\mu_2)_{N=1}\rightarrow \Rlk\mu_2\xrightarrow{N}\mu_2\rightarrow 1,$$ we obtain the isomorphism $$\ker\bigl(H^1(k, (\Rlk\mu_2)_{N=1})\rightarrow H^1(k,
\Rlk\mu_2)\bigr)\simeq\text{coker}\bigl(\mu_2(L)\xrightarrow{N}\mu_2(k)\bigr).$$ We
see that each $O(U)(k)$-orbits breaks up into one or two
$\SO(U)(k)$-orbit depending on whether $f(x)$ has an odd degree
factor or not, respectively.

We next describe the set of $\PO(U)(k)$-orbits on $V_f(k)$. Each such orbit
breaks up into either one or two $\PSO(U)(k)$-orbits depending on
whether the norm map $N:\mbox{Res}_{L/k}\mu_2/\mu_2(k) \rightarrow
\mu_2(k)$ is surjective or not, respectively (see Proposition
\ref{prop:poonenover2} for a more descriptive criterion).
As the stabilizer subscheme of $T_f$ in $\PO(U)$ is $\Rlk\mu_2/\mu_2,$
we have the following diagram of exact rows:
\begin{displaymath}
\xymatrix{
H^1(k,\mu_2)\ar[r]\ar[d]^{=}&H^1(k,\Rlk\mu_2)\ar[r]\ar[d]&H^1(k,\Rlk\mu_2/\mu_2)\ar[r]\ar[d]&H^2(k,\mu_2)\ar[d]^{=}\\
H^1(k,\mu_2)\ar[r]& H^1(k,\tO(U))\ar@{^{(}->}[r] & H^1(k,\PO(U))\ar[r] & H^2(k,\mu_2).}
\end{displaymath}
Suppose $$c'_T\in\ker(H^1(k, \Rlk\mu_2/\mu_2)\rightarrow H^1(k,
\PO(U))).$$ Since $c'_T$ maps to 0 in $H^2(k,\mu_2),$ it is the
image of some $\al\in L^\times/L^{\times2}$ well-defined up to
$k^\times/k^{\times2}.$ Since the map $H^1(k,\tO(U))\rightarrow
H^1(k,\PO(U))$ is injective, the image of $\al$ in $H^1(k,\tO(U))$ is
trivial. By Theorem \ref{thm:Oorbits}, this is equivalent to the form
$<,>_\al$ being split with discriminant 1. Therefore, we have the following
characterization of $\PO(U)(k)$-orbits.

\begin{theorem}\label{thm:POorbits}
  There is a bijection between $\PO(U)(k)$-orbits and classes
  $\al\in(L^\times/L^{\times2}k^\times)_{N=1}$ such that $<,>_\al$ is
  split. The distinguished orbit corresponds to $\al=1$.  Two ${\rm
    O}(U)(k)$-orbits corresponding to $\alpha_1,\alpha_2\in
  (L^\times/L^{\times 2})_{N=1}$ are $\PO(U)(k)$-equivalent if and
  only if $\alpha_1$ and $\alpha_2$ have the same image in
  $(L^\times/L^{\times2}k^\times)_{N=1}$.
\end{theorem}

\subsection{Connection to hyperelliptic curves}
Let $C$ be the hyperelliptic curve of genus $n$ given by the affine
equation $y^2=f(x)$, and let $J$ denote its Jacobian.
The curve $C$ has two rational points above infinity, denoted by
$\infty$ and $\infty'$. Let $P_1,\ldots,P_{2n+2}$ denote
the Weierstrass points of $C$ over $k^s$. These form the ramification
locus of the map $x:C\rightarrow \bbp^1.$ Let $D_0$ denote the hyperelliptic class obtained as
the pullback of $\sco_{\bbp^1}(1).$ Then the group $J[2](k^s)$
is generated by the divisor classes $(P_i)+(P_j)-D_0$ for $i\neq j$
subject only to the condition
that $$\sum_{i=1}^{2n+2}(P_i)-(n+1)D_0\sim 0.$$ We have the following isomorphisms of group schemes over
$k$:
\begin{equation}\label{eq:Sentence1}
J[2]\simeq(\mbox{Res}_{L/k}\mu_2)_{N=1}/\mu_2\simeq \tStab_G(T_f).
\end{equation}
An explicit formula for this identification is given in \cite[Remark 2.6]{Jerry}.

In conjunction with (\ref{eqgtoh}), this identification yields a bijection
$$
G(k)\backslash V_f(k)\longrightarrow \ker(H^1(k,J[2])\to H^1(k,G)).
$$
Thus $G(k)$-orbits on $V_f(k)$ can be identified with a subset of
$H^1(k,J[2])$. Recall that we have the following descent exact sequence:
\begin{equation}\label{eq:descentsequencetwo}
1\to J(k)/2J(k)\to H^1(k,J[2])\to H^1(k,J)[2]\to 1.
\end{equation}

A $G(k)$-orbit in $V_f(k)$ is said to be {\it soluble} if it corresponds to
a class in $H^1(k,J[2])$ which is in the image of the map from $J(k)/2J(k)$. The
following theorem states that there is a bijection between soluble
$G(k)$-orbits in $V_f(k)$ and elements of $J(k)/2J(k)$.
\begin{theorem}\label{thm:Sentence2}
  The following composite map is trivial:
  \begin{equation}
J(k)/2J(k)\to H^1(k,J[2])\to H^1(k,G).
  \end{equation}
Therefore, there is a bijection between soluble
$G(k)$-orbits in $V_f(k)$ and elements of $J(k)/2J(k)$.
\end{theorem}
\begin{proof}
We only prove the theorem in the case when $k$ is a local field. For a complete proof, see \S 3.
  Combining the descent sequence \eqref{eq:descentsequencetwo} and the
  long exact sequence obtained by taking Galois cohomology of the
  short exact sequence
$$1\rightarrow J[2]\rightarrow \Rlk\mu_2/\mu_2\xrightarrow{N} \mu_2\rightarrow 1,$$
we get the following commutative diagram.
\begin{equation}\label{eq:HUGEdiagram}
  \xymatrix{
    \lrg{(\infty')-(\infty)}\ar[r]\ar[d]^{\sim} & J(k)/2J(k) \ar@{^{(}->}[d]^{\delta} \ar[r]^{\delta'} &L^\times/L^{\times2}k^\times \ar@{^{(}->}[d]\ar[r]^{\hspace{20pt}N}& k^\times/k^{\times2}\ar[d]^{\sim}\\
    \frac{\mu_2(k)}{N(Res_{L/k}\mu_2/\mu_2(k))}\ar[r] & H^1(k, J[2])\ar[r]& H^1(k, \Rlk\mu_2/\mu_2)\ar[r]^{\hspace{20pt}N} & H^1(k, \mu_2)}
\end{equation}
The map $\delta'$ is defined in \cite{PoonenShaefer} by evaluating
$(x-\be)$ on a given divisor class. As shown in \cite{PoonenShaefer},
the first row is not exact: the image of $\delta'$ lands inside,
generally not onto, $(L^\times/L^{\times2}k^\times)_{N=1}$ with kernel
the subgroup generated by the class $(\infty')-(\infty).$ Note that $(\infty')-(\infty)\in
2J(k)$ if and only if the norm map $N:\mbox{Res}_{L/k}\mu_2/\mu_2(k) \rightarrow
\mu_2(k)$ is surjective which happens when there is a unique distinguished orbit.

To prove Theorem \ref{thm:Sentence2}, it suffices to show that if
$\al\in(L^\times/L^{\times2}k^\times)_{N=1}$ lies in the image of
$\delta',$ then $<,>_\al$ is split. We will prove this by explicitly
writing down a $k$-rational $n+1$ dimensional isotropic subspace in the special case when $k$ is a local field. For
a complete and more conceptual proof using pencils of quadrics, see Section
\ref{sec:pencil}. Suppose $\al=\delta'([D])$ for some $[D]\in
J(k)/2J(k)$ of the form
$$[D]=(Q_1)+\cdots+(Q_m)-m(\infty)\mod{2J(k)\cdot\langle(\infty')-(\infty)\rangle},$$
where $Q_1,\ldots,Q_m\in C(k^s)$ are non-Weierstrass
non-infinity points and $m\leq n+1$. When $k$ is a local field, every $[D]\in J(k)/2J(k)$ can be written in this form (\cite[Lemma 3.8]{Jerrythesis}).
If we write $Q_i=(x_i,y_i)$, then $\al = (x_1-\be)\cdots(x_m-\be)$ and
$$<\lambda,\mu>_\al = \tTr_{L/k}((x_1-\be)\cdots(x_m-\be)\lambda\mu/f'(\be)).$$
Write $$V = \prod_{1\le i<j\le m}(x_i-x_j)$$ for the Vandermonde polynomial, and for each $i=1,\ldots,m$, define
$$q_i := \prod_{1\le j\le m,j\neq i}(x_j-x_i),\quad a_i := V / q_i,\quad h_i(t) := \frac{f(t)-f(x_i)}{t-x_i}.$$
For any $j\geq0,$ we define $$g_j(t) = \sum_{i=1}^m x_i^j a_i \frac{h_i(t)}{y_i}.$$
Then the $n+1$ plane $Y$ defined below is $k$-rational and isotropic (\cite[Lemma 2.44]{Jerrythesis}):
$$
Y:=\left\{\begin{array}{cl}
\tSpan\{1,\be,\ldots,\be^n\},&\quad\mbox{if }m=1;\\[0.1in]
\tSpan\{1,\be,\ldots,\be^{n-m'},g_0(\be),\ldots,g_{m'-1}(\be)\},&\quad\mbox{if }m=2m'\mbox{ or }m=2m'+1.
\end{array}\right.
$$
This completes the proof of Theorem \ref{thm:Sentence2} in this special case.
\end{proof}

Suppose that $k$ is a number field. Then the 2-Selmer group
$\text{Sel}_2(k,J)$ is the subgroup of $H^1(k, J[2])$ consisting of
elements whose images in $H^1(k_\nu,J[2])$ lie in the image of
$J(k_\nu)/2J(k_\nu)$ for all completions $k_\nu$ of $k$. Since the
group $G=\PSO_{2n+2}$ satisfies the Hasse principle, Theorem
\ref{thm:Sentence2} implies that the following composite is also
trivial.
$$\text{Sel}_2(k,J)\rightarrow H^1(k, J[2])\rightarrow H^1(k, G).$$
A self-adjoint operator $T\in V_f(k)$ is said to be {\it locally
  soluble} if $T$ is soluble in $V_f(k_\nu)$ for all completions
$k_\nu$ of $k$.  Equivalently, $c_T$ lies in $\text{Sel}_2(k,J).$
We have thus proven the following theorem:
\begin{theorem}[\cite{Jerrythesis}]\label{theorem:imp}
  Let $k$ be a number field, and $f$ a monic separable polynomial of degree
  $2n+2$ over $k$. There is a bijection between locally soluble
  $G(k)$-orbits on $V_f(k)$ and elements in ${\rm Sel}_2(k,J)$, where
  $J$ is the Jacobian of the hyperelliptic curve given by the equation
  $y^2=f(x)$.
\end{theorem}


\subsection{Integral orbits}
Let $f(x)\in \Q[x]$ be a degree $2n+2$ monic separable polynomial, let
$C$ be the corresponding hyperelliptic curve, and $J$ its
Jacobian. We have seen that elements in the $2$-Selmer group of $J$
are in bijection with locally soluble $G(\Q)$-orbits in $V_f(\Q)$. In
this section, our aim is to show that when $f$ has integral
coefficients, every locally soluble $G(\Q)$-orbit in $V_f(\Q)$
contains an integral representative.

We do this by working over the field $\Q_p$ and the ring
$\Z_p$. Specifically, we prove the following result:
\begin{proposition}\label{prop:soluble->integral}
  Let $p$ be a prime and let
  $f(x)=x^{2n+2}+c_1x^{2n+1}+\cdots+c_{2n+2}$ be a monic separable polynomial in
  $\Z_p[x]$ such that $2^{4i}|c_i$ in $\Z_p$ for
  $i=1,\ldots,2n+2$. Then every soluble $G(\Q_p)$-orbit in
  $V_f(\Q_p)$ contains an integral representative.
\end{proposition}

Since the group $G$ has class number $1$ over $\Q$, we immediately
obtain the following corollary:
\begin{corollary}\label{cor:imp}
  Let $f(x)=x^{2n+2}+c_1x^{2n+1}+\cdots+c_{2n+2}$ be a monic separable
  polynomial in $\Z[x]$ such that $2^{4i}|c_i$ for $i=1,\ldots,2n+2$.
  Then every locally soluble $G(\Q)$-orbit in $V_f(\Q)$ contains an
  integral representative.
\end{corollary}

We will also prove the following result, which will be important to us in \S\ref{sieve}:
\begin{proposition}\label{prop:soluble=integral}
  Let $p$ be any odd prime, and let $f(x)\in\Z_p[x]$ be a degree
  $2n+2$ monic separable polynomial such that $p^2\nmid\Delta(f)$. Then the
  $G(\Z_p)$-orbits in $V_f(\Z_p)$ are in bijection with soluble
  $G(\Q_p)$-orbits in $V_f(\Q_p)$. Furthermore, if $T\in V_f(\Z_p)$,
  then $\Stab_{G(\Z_p)}(T)=\Stab_{G(\Q_p)}(T)$.
\end{proposition}

Let $p$ be a fixed prime. We start by considering the $\tO(U)(\Z_p)$-orbits. A
self-adjoint operator $T\in V_f(\Q_p)$ is {\it integral} if it
stabilizes the self-dual
lattice $$M_0=\tSpan_{\Z_p}\{e_1,\ldots,e_{n+1},f_{n+1},\ldots,f_1\}.$$
In other words, $T$ is integral if and only if when expressed in the
standard basis \eqref{eq:stdbasis}, its entries are in $\Z_p$. In
general, a lattice $M$ is self-dual if the bilinear form restricts to
a non-degenerate bilinear form: $M\times M\rightarrow \Z_p.$ Since
genus theory implies that any two self-dual lattices are
$\tO(U)(\Q_p)$-conjugate, the rational orbit of $T$ contains an
integral representative if and only if $T$ stabilizes a self-dual
lattice.

The action of $T$ on $U$ gives $U$ the structure of a $\Q_p[x]$-module,
where $x$ acts via $T$. Since $T$ is regular, we have an isomorphism
of $\Q_p[x]$-modules: $U\simeq \Q_p[x]/f(x)=L.$ Suppose $T$ is
integral, stabilizing the self-dual lattice $M_0$. The action of $T$
on $M_0$ realizes $M_0$ as a $\Z_p[x]/f(x)$-module. Write $R$ for
$\Z_p[x]/f(x).$ Since $M_0$ is a lattice, we see that after the
identification $U\simeq L,$ $M_0$ becomes a fractional ideal $I$ for
the order $R$. The split form $Q$ on $U$ gives a split form of discriminant 1 on $L$ for
which multiplication by $\be$ is self-adjoint. Any such form on $L$
is of the form $<,>_\al$ for some $\al\in L^\times$ with
$N_{L/k}(\al)\in k^{\times2}.$ The condition that $M_0$ is self-dual
translates to saying $\al\cdot I^2\subset R$ and $N(I)^2=N(\alpha^{-1})$.

The identification $U\simeq L$ is unique up to multiplication by some
element $c\in L^\times,$ which transforms the data $(I,\al)$ to $(c\cdot I,
c^{-2}\al).$ We call two pairs $(I,\al),(I',\al')$ equivalent if there
exists $c\in L^\times$ such that $I'=c\cdot I$ and $\al'=c^{-2}\al.$
Choosing a different integral representative $T$ in an integral orbit
amounts to pre-composing the map $U\simeq L$ by an element of
$\tO(U)(\Z_p)$ which does not change the equivalence class of the pair
$(I,\al).$ Hence we have a well-defined map
\begin{equation}\label{eq:intorbit}
\tO(U)(\Z_p)\backslash V_f(\Z_p)\rightarrow \mbox{ equivalence classes of pairs }(I,\al).
\end{equation}

\begin{theorem}\label{thm:intorbits}
  There is a bijection between $\tO(U)(\Z_p)$-orbits and equivalence
  classes of pairs $(I,\al)$ such that $<,>_\al$ is split, $\al\cdot I^2
  \subset R$, and $N(I)^2=N(\alpha^{-1})$. The image of $\al$ in $(L^\times/L^{\times2})_{N=1}$
  determines the rational orbit.
\end{theorem}

\begin{proof}
  Given a pair $(I,\al)$ such that $<,>_\al$ is split and $\al.I^2 =
  R$, there exists an isometry over $\Q_p$ from $(L,<,>_\al)$ to
  $(U,<,>_Q)$ that sends $I$ to the self-dual lattice $M_0$. The image
  of the multiplication by $\be$ operator lies in $V_f(\Z_p).$ Any two
  such isometries differ by an element in $\tO(U)(\Z_p),$ hence we get a
  well-defined $\tO(U)(\Z_p)$-orbit. Along with \eqref{eq:intorbit},
  we have proved the first statement.

For the second statement, from the sequence of isometries
\eqref{eq:isoG}, we see that since $<,>_\al$ is split, there exists
$g\in\tO(U)(\Q_p^s)$ such
that $$\gl{\sqrt{\al}}{\sigma}/\sqrt{\al}=g^{-1}\gl{g}{\sigma},\quad\forall
\sigma\in\Gal(k^s/k).$$ Here, the left hand side is viewed as an
element of $\tStab_{\tO(U)}(T_f).$ The rational orbit corresponding
the pair $(I,\al)$ is therefore the rational orbit of $T=gT_fg^{-1}.$
The rest follows formally from unwinding definitions.
\end{proof}

Suppose the $\tO(U)(\Z_p)$-orbit of some $T\in V_f(\Z_p)$ corresponds
to an equivalence class of pair $(I,\al).$ Then the stabilizer of $T$
in $\GL(U)(\Z_p)$ is $\tEnd_R(I)^\times.$ Moreover, just as the proof
of Proposition \ref{prop:geoorbit}, we have
\begin{eqnarray*}
\tStab_{\tO(U)}(T)(\Z_p) &=& \tEnd_R(I)^\times[2],\\
\tStab_{\SO(U)}(T)(\Z_p) &=& (\tEnd_R(I)^\times[2])_{N=1}.
\end{eqnarray*}
The stabilizers in the group $\PO(U)(\Z_p)$ (and $\PSO(U)(\Z_p)$) are slightly complicated because $\PO(U)(\Z_p)$ contains $\tO(U)(\Z_p)$ as a subgroup with quotient $\Z_p^\times/\Z_p^{\times2}.$ We have the following exact sequences.
$$1\rightarrow \tEnd_R(I)^\times[2]/\mu_2 \rightarrow \tStab_{\PO(U)}(T)(\Z_p) \rightarrow (R^{\times2}\cap\Z_p^\times)/\Z_p^{\times2}\rightarrow 1.$$
\begin{equation}\label{eq:stabZp}
1\rightarrow (\tEnd_R(I)^\times[2])_{N=1}/\mu_2 \rightarrow \tStab_{\PSO(U)}(T)(\Z_p) \rightarrow (R^{\times2}\cap\Z_p^\times)/\Z_p^{\times2}\rightarrow 1.
\end{equation}

\vspace{.1 in}
\noindent\textbf{Proof of Proposition \ref{prop:soluble->integral}:} First note
that it suffices to show that the $\PO(U)(\Q_p)$-orbit of $T$ contains
an integral representative. Since $T$ is soluble, there exists some
$[D]\in J(\Q_p)/2J(\Q_p)$ such that
$\w{\al}=\delta'([D])\in(L^\times/L^{\times2}\Q_p^\times)_{N=1}$
corresponds to the $\PO(U)(\Q_p)$-orbit of $T$. By
\cite[Lemma 3.8]{Jerrythesis}, there exists non-Weierstrass non-infinity
points $Q_1,\ldots,Q_m\in C(\Q_p^s)$, with $m\leq n+1$, such that
\begin{equation}\label{eq:Dm}
[D]=(Q_1)+\cdots+(Q_m)-m(\infty)\mod{2J(\Q_p)\cdot\langle(\infty')-(\infty)\rangle}.
\end{equation}
Write each $Q_i=(x_i,y_i)\in C(\sco_{\Q_p^s})$ then $\al =
(x_1-\be)\cdots(x_m-\be)$ is a lift of $\w{\al}$ to $L^\times.$ We
claim that either the $\tO(U)(\Q_p)$-orbit of $T$ corresponding to the
image of $\al$ in $L^\times/L^{\times2}$ has an integral
representative, or $[D]$ can be expressed in the form \eqref{eq:Dm}
with $m$ replaced by $m-2$. Applying induction on $m$ completes the
proof.

The claim follows verbatim from the proof of \cite[Proposition 8.5]{BG3}. We give a quick sketch here. Let $r(x)\in \Q_p[x]$ be a
polynomial of degree at most $m-1$ such that for all $i$, $r(x_i)=y_i$ and
let $$p(x)=(x-x_1)\cdots(x-x_m)\in \Z_p[x].$$ Now $p(x)$ divides
$r(x)^2 - f(x)$ in $\Q_p[x]$ and let $q(x)$ denote the quotient. By definition, $\al = (-1)^mP(\be).$
If the polynomial $r(x)\in \Z_p[x],$ then the ideal $I =
(1,r(\be)/\al)$ does the job. Note $\al I^2 = (\al, r(\be),q(\be)).$
The integrality assumption of $r(x)$ is used to show that
$r(\be),q(\be)\in R.$ A computation of ideal norms shows that $N(I)^2=N(\al)^{-1}.$

When $r(x)$ is not integral, a Newton polygon analysis on $f(x) -
r(x)^2$ shows that $\tdiv(y - r(x))-[D]$ has the form $D^*+E$ with
$D^*,E\in J(\Q_p)$ where $D^*$ can be expressed in \eqref{eq:Dm} with
$m$ replaced by $m-2$ and the $x$-coordinates of the non-infinity
points in $E$ have negative valuation. The condition of divisibility
on the coefficients of $f(x)$ ensures that $E\in
2J(\Q_p).((\infty')-(\infty)),$ or equivalently $(x - \be)(E)\in
L^{\times2}\Q_p^\times.$ $\Box$\vspace{.2 in}

\noindent\textbf{Proof of Proposition \ref{prop:soluble=integral}:} Once again,
it suffices to work with $\PO(U)$-orbits instead of $\PSO(U)$-orbits
directly. The assumption on $\Delta(f)$ implies that $R$ is the
maximal order. Hence there is a bijection between
$\tO(U)(\Z_p)$-orbits and $(R^\times/R^{\times2})_{N=1}.$ Note over
non-archimedean local fields, the splitness of the quadratic form is
automatic from the existence of a self-dual lattice. Taking flat
cohomology over $\text{Spec}(\Z_p)$ of the
sequence $$1\rightarrow\mu_2\rightarrow\text{O}(U)\rightarrow
\PO(U)\rightarrow1$$
gives: $$1\rightarrow\text{O}(U)(\Z_p)/\pm1\rightarrow\PO(U)(\Z_p)\rightarrow
\Z_p^\times/\Z_p^{\times2}\rightarrow 1.$$ Hence $\PO(U)(\Z_p)$-orbits
correspond bijectively to $(R^\times/R^{\times2}\Z_p^\times)_{N=1}.$

On the other hand, the assumption on $\Delta(f)$ implies that the
projective closure $\mathcal{C}$ of the hyperelliptic curve $C$ defined by affine
equation $y^2 = f(x)$ over $\tSpec(\Z_p)$ is regular. Since the special fiber of $\mathcal{C}$ is geometrically reduced and irreducible, the Neron model $\scj$ of its Jacobian $J_{\Q_p}$ is fiberwise connected (\cite[\S9.5 Theorem 1]{BLR}) and its
2-torsion $\scj[2]$ is isomorphic to
$(\text{Res}_{R/\Z_p}\mu_2)_{N=1}/\mu_2.$ Using diagram
\eqref{eq:HUGEdiagram} after replacing $L,k,J$ by $R,\Z_p,\scj$, we
see that the vertical maps are all isomorphisms and $\delta'$ maps
$\scj(\Z_p)/2\scj(\Z_p)$ surjectively to
$(R^\times/R^{\times2}\Z_p^\times)_{N=1}$. The Neron mapping property
implies that $\scj(\Z_p)/2\scj(\Z_p)=J(\Q_p)/2J(\Q_p).$

Suppose the $\tO(U)(\Z_p)$-orbit of some $T\in V_f(\Z_p)$ corresponds
to an equivalent class of pair $(I,\al)$. Since $R$ is
maximal, $\tEnd_R(I)=R$. Since $R^\times[2] = L^\times[2],$ we see
from \eqref{eq:stabZp} that it remains to compare $(R^{\times2}\cap \Z_p^\times)/\Z_p^{\times2}$ with $(L^{\times2}\cap \Q_p^\times)/\Q_p^{\times2}.$
These two sets are only nonempty when $L$ contains a quadratic extension $K'$ of $\Q_p$. The condition $p^2\nmid\Delta(f)$ implies that $K'=\Q_p(\sqrt{u})$ can only be the unramified quadratic extension of $\Q_p$. In other words, $u\in\Z_p^\times.$ Hence in this case $(L^{\times2}\cap \Q_p^\times)/\Q_p^{\times2}$ and $(R^{\times2}\cap \Z_p^\times)/\Z_p^{\times2}$ both are equal to the group of order 2 generated by the class of $u$.
$\Box$\vspace{.1 in}

\section{Interpretation using pencils of quadrics}
\label{sec:pencil}

In this section, we give geometric meanings to the notion of
distinguished and soluble. For the proof of all the statements below,
see \cite[Section 2.2]{Jerry}. These geometric interpretations are not
necessary if one wants only the average size of the 2-Selmer groups.

Let $k$ be a field of characteristic not 2 and let $f(x)$ be a
monic separable polynomial of degree $2n+2$. Let $T$ be a self-adjoint
operator in $V_f(k)$ and let $C$ denote the hyperelliptic curve
$y^2=f(x)$. Let $\infty$ and $\infty'$ denote the two points above
infinity. One has a pencil of quadrics in $U$ spanned by the following
two quadrics:
\begin{eqnarray*}
Q(v) &=& \qu{v}{v}\\
Q_T(v) &=& \qu{v}{Tv}.
\end{eqnarray*}
This pencil is generic in the sense that there are precisely $2n+2$
singular quadrics among $x_1Q - x_2Q_T$ for $[x_1,x_2]\in \bbp^1,$ and
that they are all simple cones. Its associated hyperelliptic curve
$C'$ is the curve parameterizing the rulings of the quadrics in the
pencil. A ruling of a quadric $Q_0$ is a connected component of the
Lagrangian variety of maximal isotropic subspaces. When $Q_0$ is a
simple cone, there is only one ruling. When $Q_0$ is non-degenerate,
there are two rulings defined over $k(\sqrt{\disc(Q_0)}).$ To give a
point on $C'$ is the same as giving a quadric in the pencil along with
a choice of ruling. Therefore, the curve $C'$ is isomorphic
non-canonically to the hyperelliptic curve
$$y^2 = \text{disc}(xQ - Q_T) = \text{disc}(Q)\det(xI - T)=f(x).$$
Hence $C'$ is isomorphic to $C$ over $k$. We fix an isomorphism
$C'\simeq C$ and denote by $Y_0$ the ruling on $Q$ that corresponds to
$\infty\in C(k)$. Since $C$ has a rational point, the Fano variety
$F_T$ of $n$-planes isotropic with respect to both quadrics is a
torsor of $J$ of order dividing 2. In fact, it fits inside a
disconnected algebraic group $$J\dcup F_T\dcup \pico(C)\dcup F'_T$$ where
$F'_T\simeq F_T$ as varieties. Using the point $\infty,$ one obtains a
lift of $F_T$ to a torsor of $J[2]$ by taking
\begin{eqnarray*}
F_T[2]_\infty &=& \{X\in F_T|X + X = (\infty)\}\\
&=&\{X\,\,n\mbox{-plane}|\tSpan\{X,TX\}\mbox{ is an isotropic }n+1\mbox{ plane in the ruling }Y_0\}.
\end{eqnarray*}
The second equality is \cite[Proposition 2.32]{Jerry}.

The group scheme $G=\PSO(U)$ acts on the $k$-scheme $$W_f=\{(T,X)|T\in
V_f, X\in F_T[2]_\infty\}$$ via $g.(T,X)=(gTg^{-1},gX).$ Let $W_T$
denote the fiber above any fixed $T\in V_f(k)$. This action is
simply-transitive on $k$-points (\cite{Jerry} Corollary 2.36). Hence for any $T\in V_f(k)$, the above action induces a simply-transitive action of
$J[2]\simeq\tStab_G(T)$ on the fiber $W_T=F_T[2]_\infty.$

\begin{theorem}\label{thm:theoremfrompencils} {\rm (\cite[Proposition 2.38]{Jerry}, \cite[Lemma 2.19]{Jerrythesis})}
These two actions of $J[2]$ coincide and as elements of $H^1(k, J[2]),$
\begin{equation}\label{eq:FTcT}
[F_T[2]_\infty] = [W_T] = c_T.
\end{equation}
\end{theorem}

For hyperelliptic curves with a rational Weierstrass point, one can
obtain all torsors of $J[2]$ using pencils of quadrics
(\cite[Proposition 2.11]{Jerrythesis}). For hyperelliptic curves with
no rational Weierstrass point but with a rational non-Weierstrass
point, we do not recover all torsors of $J[2]$ using pencils of
quadrics but we recover enough to study $\PSO(U)(k)$-orbits.

Suppose $T\in V_f(k)$. From \eqref{eq:FTcT}, we see that there exists
a $k$-rational $n$-plane $X$ such that $\tSpan\{X,TX\}$ is an
isotropic $n+1$ plane if and only if either $[F_T[2]_\infty]$ or
$[F_T[2]_{\infty'}]$ is trivial. Again by \eqref{eq:FTcT}, this is
equivalent to $c_T$ being in the image of the subgroup generated by
$(\infty')-(\infty)\in J(k)/2J(k)$ under the descent map
$J(k)/2J(k)\hookrightarrow H^1(k,J[2]).$ Commutativity of the top left
square in \eqref{eq:HUGEdiagram} implies that this is in turn
equivalent to $c_T$ mapping to $0$ in $H^1(k, \tStab_{\PO(U)}(T)).$
Finally, this is equivalent to $T$ being distinguished. We have
therefore proved Proposition \ref{prop:Wt}.

Since $[F_T[2]_\infty]$ maps to $[F_T]$ under the canonical map
$H^1(k, J[2])\rightarrow H^1(k,J)[2],$ we see that $T$ is soluble if
and only if $F_T(k)\neq\emptyset$. This equivalence of solubility and
the existence of rational points is the main reason why the name
``soluble'' is used. Likewise, $T$ is locally soluble if and only if
$F_T(k_\nu)\neq\emptyset$ at all places~$\nu.$

We now give a complete proof for the claim that if
$\al\in(L^\times/L^{\times2}k^\times)_{N=1}$ lies in the image of
$\delta',$ then $<,>_\al$ is split. Consider instead the pencil of
quadrics in $L$ spanned by the following two quadrics:
\begin{eqnarray*}
Q_\al(\lambda)&=&<\lambda,\lambda>_\al\\
Q'_\al(\lambda)&=&<\lambda,\be\lambda>_\al.
\end{eqnarray*}
This pencil is once again generic, its associated hyperelliptic curve $C_\al$ is smooth of genus $n$ isomorphic non-canonically to the hyperelliptic curve defined by affine equation
$$y^2 = \text{disc}(xQ_\al-Q'_\al) = N_{L/k}(\al)f(x).$$
Since $N_{L/k}(\al)\in k^{\times2},$ the curve $C_\al$ is isomorphic
to $C$ over $k$. Fix any isomorphism $C'_\al\simeq C.$ The Fano
variety $F_\al$ of $n$-planes isotropic with respect to both quadrics
is a torsor of $J$ of order dividing $2$. There are two natural lifts
of $F_\al$ to torsors of $J[2]$ by taking
$$F_\al[2]_\infty = \{X\in F|X + X = (\infty)\}\quad\mbox{or}\quad F_\al[2]_{\infty'} = \{X\in F|X + X = (\infty')\}.$$
As elements of $H^1(k, J[2]),$ these two lifts map to the same class
in $H^1(k, \Rlk\mu_2/\mu_2).$ The class $\al$ also maps to a class in
$H^1(k, \Rlk\mu_2/\mu_2)$ as in \eqref{eq:HUGEdiagram}. By \cite[Proposition 2.27]{Jerrythesis}, these two classes coincide. When
$\al=\delta'([D])$ comes from $J(k)/2J(k),$ one of these two lifts
recovers $[D]$ and hence $F_\al(k)\neq\emptyset.$ Pick any $X\in
F_\al(k).$ If $X + X = (\infty),$ then $[D]=0,\al=1$ and $<,>$ is
split. Otherwise, $\tSpan\{X, (\infty)-X\}$ is a $k$-rational $n+1$
plane isotropic with respect to $<,>_\al.$

\section{Orbit counting}
\label{sec:orbitcounting}

In this section, we let the monic polynomial $f$ vary and count the
average number of locally soluble orbits of the action of $G(\Q)$ on
$V_f(\Q)$.
We redefine $V$ to be 
the following scheme over $\Z$:
$$V = \{T:U\rightarrow U|T = T^*, \text{Trace}(T)=0\}\simeq\A^{2n^2 +
  5n + 2}_\Z.$$ For any ring $R$, we shall think of elements in $V(R)$
as $B=AT$, where $A$ is the matrix with $1$'s on the anti-diagonal and $0$'s elsewhere and where $T$ is a
$(2n+2)\times(2n+2)$ matrix with coefficients in $R$ such that
Trace($T)=0$ and $T = T^*$. Thus, elements $B\in V(R)$ are symmetric
matrices with anti-trace $0$. This change of perspective is only to
simplify notation in what follows. The group scheme $G=\PSO_{2n+2}$
acts on $V$ by $g\cdot B:=gBg^t$. The ring of polynomial
invariants for this action is generated by the coefficients
$c_2,\ldots,c_{2n+2}$ of the polynomial $\det(Ax-By)$. We
define the scheme $S$ to be:
$$S=\tSpec\,\Z[c_2,\ldots,c_{2n+2}].$$ The map $\pi:V\rightarrow S$ is
given by the coefficients of the characteristic polynomial; we call
$\pi(B)$ the invariant of $B$.

A point $c=(c_2,\ldots,c_{2n+2})\in S(\R)$ corresponds to a monic
polynomial $$f_c(x):=f(x) = x^{2n+2}+c_2x^{2n}+\cdots+c_{2n+2}.$$ We define
its height $H(f)$ by $$H(f):=H(c):=\max\{|c_k|^{d/k}\}_{k=2}^{2n+2},$$
where $d=(2n+2)(2n+1)=\deg H$ is the ``total degree'' of the discriminant of
$f$. The height of $B\in V(\R)$ is defined to be the height of
$\pi(B)$, and the height of the hyperelliptic curve $C(c)$ given by
$y^2=f(x)$ is defined to be $H(c)$.

For each prime $p$, let $\Sigma_p$ be a closed subset of
$S(\Z_p)\backslash\{\Delta=0\}$ whose boundary has measure $0$. Let
$\Sigma_\infty$ be the set of all $c\in S(\R)\backslash\{\Delta=0\}$
such that the corresponding polynomial $f$ has $m$ distinct pairs of
complex conjugate roots, where $m$ belongs to a fixed subset of
$\{0,\ldots,n+1\}$. To such a collection $(\Sigma_\nu)_\nu$, we
associate the family $F=F_\Sigma$ of hyperelliptic curves (with a
marked rational non-Weierstrass point), where $C(c)\in F$ if and only
if $c\in\Sigma_\nu$ for all places $\nu$. Such a family is said to be
{\it defined by congruence conditions}.

Given a family $F$ that is defined by congruence conditions, let
$\Inv(F)\subset S(\Z)$ denote the set $\{c(C):C\in F\}$ of invariants. We denote the
$p$-adic closure of $\Inv(F)$ in $S(\Z_p)\backslash\{\Delta=0\}$ by
$\Inv_p(F)$. We say that a family $F$ defined by congruence conditions
is {\it large at $p$} if $\Inv_p(F)$ contains every element $c\in
S(\Z_p)$ such that $p^2\nmid\Delta(c)$. Finally, we say that $F$ and
$\Inv(F)$ are {\it large} if $F$ is large at all but finitely many
primes.  An example of a large subset of $S(\Z)$ is the
set $$F_0=\{(c_2,\ldots,c_{2n+2})\in S(\Z)|p^k\nmid c_k,\forall
k=2,\ldots,2n+2\}.$$ Another example is the set of elements in $S(\Z)$
having squarefree discriminant.


Our goal is to prove the following theorem:
\begin{theorem}\label{thm:largegood}
  The average number of locally soluble orbits for the action of
  $G(\Q)$ on $V_f(\Q)$ as $f$ runs through any large subset of
  $S(\bbz),$ when ordered by height, is 6.
\end{theorem}

In view of the correspondence (in Theorem \ref{theorem:imp}) between locally
soluble orbits and 2-Selmer elements, the above result immediately
implies the following strengthening of Theorem \ref{thm:evenmainssss}:

\begin{theorem}\label{thm:largegoodcurve}
When all hyperelliptic curves over $\bbq$ of genus $n$ with a marked
rational non-Weierstrass point in any large family are ordered by
height, the average size of the 2-Selmer group of their Jacobians is
6.
\end{theorem}

\subsection{Outline of the proof}
We now give an outline of the proof of Theorem
\ref{thm:largegood}. Let $F$ be a large subset of $S(\bbz).$ Since the
curve $C(c_2,\ldots,c_{2n+2})$ is isomorphic to
$C(u^2c_2,\ldots,u^{2n+2}c_{2n+2})$, for any $u\in\Q$, we may assume
that $2^{4i}\mid c_i$ for every $(c_2,\ldots,c_{2n+2})\in\Inv(F).$
Hence by Corollary \ref{cor:imp}, it suffices to
determine the average number of locally soluble $G(\Q)$-equivalence
classes on $V_f(\Z)$, as $f$ runs through $F$.

As a first step, we count the number of $\R$-soluble $G(\Z)$-orbits of
$V(\Z)$ having bounded height and non-zero discriminant.  An element
in $V(\bbz)$ is \textit{reducible} if either the discriminant of its
characteristic polynomial is 0 or it is distinguished; otherwise it is
called \textit{irreducible}. Apart from a negligible number of
invariants $c\in S(\Z)$ (Proposition \ref{lem:100distinguish}), there will
always be 2 distinguished orbits having invariant $c$.  Let
$V(\R)^\sol$ denote the set of $\R$-soluble elements of $V(\bbr).$ To
estimate the number of irreducible orbits having bounded height, we
construct in Section \ref{sec:fundomain} a fundamental domain for the
action of $G(\bbz)$ on $V(\R)^\sol$.

The difficulty in estimating the number of lattice points in this
fundamental domain is that it is not compact, but rather has cusps
going to infinity. We handle these cusps by averaging this fundamental
domain over a bounded subset of $G(\R)$, and breaking it up into two
pieces, namely, the main body and the cusp region.
We show in
Section \ref{sec:cusp} that the cusp region has small volume and
negligibly many irreducible elements while the main body has a small number of
reducible elements. Hence, using Proposition \ref{davlem}, we obtain:
\begin{equation}\label{eqoutlinefirst}
\#(V(\Z)^\irr\cap V(\R)^\sol_{<X})\sim\Vol(G(\Z)\backslash V(\R)^\sol_{<X}),
\end{equation}
where $V(\Z)^\irr$ denotes the set of irreducible elements,
$V(\R)^\sol_{<X}$ denotes the set of points in $V(\R)$ that are
$\R$-soluble and have height less than $X$, and the above volume is
taken with respect to Euclidean measure $\nu$ on $V$ normalized so that
$V(\Z)\subset V(\R)$ has co-volume $1$. In other words, the number of
irreducible integral orbits that are soluble at $\bbr$ of height less
than $X$ is asymptotic to the volume of a fundamental domain for the
action of $G(\Z)$ on $V(\R)^\sol_{<X}$.

Fix $\tau$ and $\mu$ to be Haar measures on $G(\bbr)$ and $S(\bbr)$,
respectively, induced from left-invariant differential top forms over
$\Q$ where $\mu$ is normalized such that $S(\Z)\subset S(\bbr)$ has
co-volume $1$. For suitably ``nice'' morphisms $\delta: G\times S\to V$, there exists a
fixed rational constant $\J$ such that
\begin{equation}\label{eqmeasurechange}
\delta^*d\nu=\J\cdot d\tau\wedge d\mu.
\end{equation}
Here, $\J$ is independent of $\delta$.

Let $S(\R)_{<X}$ denote
the set of invariants $c\in S(\bbr)$ of height less than $X$.
For
any place $\nu$ of $\bbq$, let $a_\nu$ be the ratio
\begin{equation}\label{eq:a_nu}
a_\nu=\frac{|J(\bbq_\nu)/2J(\bbq_\nu)|}{|J(\bbq_\nu)[2]|}.
\end{equation}
Here $J$ is the Jacobian of any hyperelliptic curve of genus $n$. The
above quotient depends only on $\Q_\nu,n$ (\cite[Lemmas 5.7,
5.14]{Stoll01}) and satisfies the product formula $\prod_\nu a_\nu=1$.
We use (\ref{eqmeasurechange}) to compute the right hand side of
(\ref{eqoutlinefirst}) obtaining:
\begin{equation}\label{eqoutlinesecond}
\begin{array}{rcl}
\#(V(\Z)^\irr\cap V(\R)^\sol_{<X})&\sim&\Vol(G(\Z)\backslash V(\R)^\sol_{<X})\\[0.1in]&=&
|\J|\cdot a_\infty\cdot\tau(G(\Z)\backslash G(\R))\cdot\mu(S(\R)_{<X}).
\end{array}
\end{equation}

To prove Theorem \ref{thm:largegood}, we need to instead count $G(\Q)$-equivalence
classes of locally soluble elements of $V(\Z)$ having invariants in
$\Inv(F)$. We accomplish this via a sieve in Section \ref{sieve}. The vital
ingredient for this sieve is a uniformity estimate proved in
\cite{geosieve}. The sieving factor at the finite places are computed
in Section \ref{secvol} to be
\begin{equation}\label{eq:localfactor}
|\J|_p\cdot a_p\cdot\tau(G(\bbz_p))\cdot \mu_p,
\end{equation}
where $\mu_p$ is the local density of $\Inv(F)$ at $p$, namely,
$\mu_p:=\mu(\Inv_p(F))/\mu(S(\Z_p))$. The analogous local density at
$\infty$ is given by
$\mu_\infty:=\mu(\Inv_\infty(F)_{<X})/\mu(S(\R)_{<X})$.

Therefore, we finally obtain:
\begin{equation}
\begin{array}{rcl}
\#(G(\Q)\backslash V_F^{{\rm ls, irr}}(\Q)_{<X})&\sim&|\J|\cdot
a_\infty\cdot\tau(G(\Z)\backslash
G(\R))\cdot\mu(S(\R)_{<X})\cdot\mu_\infty\displaystyle\prod_p
\left(|\J|_p\cdot a_p\cdot\tau(G(\bbz_p))\cdot \mu_p\right)\\
&\sim&\tau_G\mu(S(\R)_{<X})\mu_\infty\prod_p\mu_p,
\end{array}
\end{equation}
where $G(\Q)\backslash V_F^{{\rm l s,\irr}}(\Q)_{<X}$ denotes a set of
representatives for the $G(\Q)$-equivalence classes of locally soluble irreducible
elements in $V(\Q)$ having invariants in $\Inv(F)$ and height bounded
by $X$.

We will show that up to a negligible quantity, the number of
hyperelliptic curves $C:y^2=f(x)\in F$ is equal to
$\mu(S(\R)_{<X})\mu_\infty\prod_p\mu_p$. Furthermore, for $100\%$ of
these curves, the set $V_f(\Q)$ contains two distinct distinguished
orbits.  Thus, the average number of locally soluble orbits for the
action of $G(\Q)$ on $V_f(\Q)$ is equal to $2+\tau_G=6$.

\subsection{Construction of fundamental domains}\label{sec:fundomain}
Let $V(\R)^\sol$ denote the set of $\R$-soluble elements in $V(\R)$
having nonzero discriminant. We partition $V(\R)^\sol$ into $n+2$
sets as follows,
$$
V(\R)^\sol=\bigcup_{m=0}^{n+1}V(\R)^{(m)},
$$
where $V(\R)^{(m)}$ consists of elements $B\in V(\R)^\sol$ such that
the polynomial corresponding to $\pi(B)$ has $m$ pairs of complex
conjugate roots (and $2n+2-2m$ real roots). In this section, our goal
is to describe convenient fundamental domains for the action of
$G(\Z)$ on $V(\R)^{(m)}$ for $m\in\{0,\ldots,n+1\}$.

\subsubsection*{Fundamental sets for the action of $G(\R)$ on $V(\R)^\sol$}
First, we construct convenient fundamental sets for the action of
$G(\R)$ on $V(\R)^{(m)}$. Let $S(\R)^{(m)}$ denote the set of elements
$c\in S(\R)\backslash\{\Delta =0\}$ such that the corresponding
polynomial has $m$ pairs of complex conjugate roots. There exists an
algebraic section $\kappa:S\to V$ defined over $\Z[1/2]$ such that
every element in the image of $S(\R)\backslash\{\Delta=0\}$ under
$\kappa$ is distinguished \cite[Section 3.1]{Jerrythesis}.  The number
of $\R$-soluble $G(\R)$-orbits in $V_{f_c}(\R)$, for $c\in S(\R)^{(m)}$,
depends only on $m$. We denote it by $\tau_m$. There exist elements
$g_1,\ldots g_{\tau_m}\in \GL(U)(\R)$ such that the set
\begin{equation}\label{eqfirstfundset}
R'^{(m)}:=\bigcup_i g_i\kappa(S(\R)^{(m)})g_i^{-1}
\end{equation}
is a fundamental set for $G(\R)\backslash V(\R)^{(m)}$.
Indeed, since $L:=\R[x]/f_c(x)$ is independent of $c\in S(\R)^{(m)}$,
an element $g\in \GL(U)(\R)$ that conjugates $\kappa(c_0)$ to a
$G(\R)$-orbit corresponding to a class $\al\in (L^\times/L^{\times
  2}\R^\times)_{N=1}$ does so for every $c\in S(\R)^{(m)}$.

We now construct our fundamental set $R^{(m)}$ for $G(\R)\backslash V(\R)^{(m)}$ to be
\begin{equation}\label{eqfundsetfinal}
  R^{(m)}:=\R_{>0}\cdot\{B\in R'^{(m)}:H(B)=1\}.
\end{equation}
The reason we use the set $R^{(m)}$ instead of $R'^{(m)}$ is that the
size of the coefficients of each element in $R^{(m)}$ having height
$X$ is bounded by $O(X^{1/d})$, where $d=(2n+2)(2n+1)$ is the degree
of the height function. This follows because the elements in
$R'^{(m)}$ having height 1 lie in a bounded subset of $V(\R)$.

\subsubsection*{Fundamental domains for the action of $G(\Z)$ on $G(\R)$}
We now describe Borel's construction \cite{Borel} of a fundamental
domain $\FF$ for the left action of $G(\Z)$ on $G(\R)$.  Let $G(\R)=NTK$ be
the Iwasawa decomposition of $G(\R)$. Here, $N\subset G(\R)$ denotes
the set of unipotent lower triangular matrices, $T\subset G(\R)$
denotes the set of diagonal matrices, and $K\subset G(\R)$ is a
maximal compact subgroup. Then a fundamental domain $\FF$ for the action of
$G(\Z)$ on $G(\R)$ may be expressed in the following form:
$$
\FF:=\{utk:u\in N'(t),\;t\in T',\;k\in K\}\subset N'T'K
$$
where $N'\subset N$ is a bounded set, $N'(t)\subset N'$ is a measurable set
depending on $t\in T'$, and $T'\subset T$ is given by
$$
T':=\{{\rm diag}(t_1^{-1},t_2^{-1},\ldots,t_{n+1}^{-1},t_{n+1},\ldots,t_1):t_1/t_2>c,\ldots,t_n/t_{n+1}>c,t_nt_{n+1}>c\},
$$
for some constant $c>0$.

\subsubsection*{Fundamental domains for the action of $G(\Z)$ on $V(\R)^\sol$}
For $h\in G(\R)$, we regard $\FF h\cdot R^{(m)}$ as a multiset, where
the multiplicity of $B$ in $\FF h\cdot R^{(m)}$ is given by
$\#\{g\in\FF:B\in gh\cdot R^{(m)}\}$.  The $G(\Z)$-orbit of any $B\in
V(\R)$ is represented $\#\Stab_{G(\R)}(B)/\#\Stab_{G(\Z)}(B)$ times in
this multiset $\FF h\cdot R^{(m)}$.

The group $\Stab_{G(\Z)}(B)$ is nontrivial only for a measure $0$ set
in $V(\R)^{(m)}$. Indeed, $G(\Z)$ is countable and every element $g\in
G(\Z)$ only fixes a measure 0 set in $V(\R)$. (Later on, in Proposition
\ref{prop:bigstabsmall}, we will show that the number of $G(\Z)$-orbits
on $V(\Z)$ having a nontrivial stabilizer in $G(\Z)$ is negligible.)
The size $\#\Stab_{G(\R)}(B)$ is constant over $B\in V(\R)^{(m)}$. We
denote it by $\#J^{(m)}[2](\R)$.  Therefore, the multiset $\FF h\cdot
R^{(m)}$ is a cover of a fundamental domain for $G(\Z)$ on
$V(\R)^{(m)}$ of degree $\#J^{(m)}[2](\R)$.

\subsection{Averaging, cutting off the cusp, and estimation in the main body}\label{sec:cusp}
An element $B\in V(\Q)$ is said to be {\it irreducible} if it has
nonzero discriminant and it is not distinguished.  For any
$G(\Z)$-invariant set $S\subset V(\Z)^{(m)}:=V(\R)^{(m)}\cap V(\Z)$,
let $N(S;X)$ denote the number of irreducible $G(\Z)$-orbits of $S$
that have height bounded by $X$, where each orbit $G(\Z)\cdot B$ is weighted by $1/\#\Stab_{G(\Z)}(B)$. The result of the previous section
shows that we have
$$
N(S;X)=\frac{1}{\#J^{(m)}[2](\R)}\#\{\FF hR^{(m)}(X)\cap S^\irr\}
$$
for any $h$ in $G(\R)$, where $R^{(m)}(X)$ denotes the elements in
$R^{(m)}$ having height bounded by $X$ and $S^\irr$ denotes the set of
irreducible elements in $S$.  Let $G_0$ be a bounded open
$K$-invariant ball in $G(\R)$. Averaging the above equation over $h\in
G_0$ we obtain:
\begin{equation}\label{eqavgone}
  N(S;X)= \displaystyle\frac{1}{\#J^{(m)}[2](\R)\Vol(G_0)}\displaystyle\int_{h\in G_0}\#\{\FF hR^{(m)}(X)\cap S^\irr\}dh,
\end{equation}
for any Haar-measure $dh$ on $G(\R)$, and where the volume of $G_0$ is
computed with respect to $dh$. We use \eqref{eqavgone} to define
$N(S;X)$ when $S$ is not $G(\Z)$-invariant.

By an argument identical to the proof of \cite[Theorem 2.5]{BS}, we
obtain
\begin{equation}\label{eqavgtwo}
N(S;X)=\frac{1}{\#J^{(m)}[2](\R)\Vol(G_0)}\displaystyle\int_{h\in\FF}\#\{hG_0R^{(m)}(X)\cap S^\irr\}dh.
\end{equation}

To estimate the number of integral points in the bounded region
$hG_0R^{(m)}(X)$, we use the following result of
Davenport~\cite{Davenport1}.
\begin{proposition}\label{davlem}
  Let $\mathcal R$ be a bounded, semi-algebraic multiset in $\R^n$
  having maximum multiplicity $m$, and that is defined by at most $k$
  polynomial inequalities each having degree at most $\ell$.
  Then the number of integral lattice points $($counted with
  multiplicity$)$ contained in the region $\mathcal R$ is
\[\Vol(\mathcal R)+ O(\max\{\Vol(\bar{\mathcal R}),1\}),\]
where $\Vol(\bar{\mathcal R})$ denotes the greatest $d$-dimensional
volume of any projection of $\mathcal R$ onto a coordinate subspace
obtained by equating $n-d$ coordinates to zero, where
$d$ takes all values from
$1$ to $n-1$.  The implied constant in the second summand depends
only on $n$, $m$, $k$, and $\ell$.
\end{proposition}

We can express any $h\in \FF$ as $h=utk$, where $u\in N'$, $t\in T'$, and
$k\in K$. Since $G_0$ is $K$-invariant, we have for any $h\in\FF$,
$$hG_0R^{(m)}(X) =
utkG_0R^{(m)}(X) = t(t^{-1}ut)G_0R^{(m)}(X).$$ By the descriptions of
$N'$ and $T'$, we see that the set $t^{-1}N't$ is bounded independent
of $t\in T'$.  (The coordinates of elements in $N'$ are scaled by
either $(t_i/t_{i+1})^{-1}$ for $i=1,\ldots,n,$ or
$(t_nt_{n+1})^{-1}.$ which are bounded above by $1/c'$.)  Therefore
$(t^{-1}ut)G_0R^{(m)}(X)$ is a compact region where the
coefficients of the elements inside are growing homogeneously in
$X$. It is the action of $t\in T'$ that stretches and compresses
different coordinates.

As $t$ grows in $T'$, the estimates on the number of integral points
in $hG_0R^{(m)}(X)$ obtained from Proposition \ref{davlem} gets worse
and worse. Indeed when $t$ gets high enough (in the cusp of $T'$), the
top left entry $b_{11}$ of every element in $hG_0R^{(m)}(X)$ will be
less than 1 in absolute value, at which point the error term in
Proposition \ref{davlem} dominates the main term. As $t$ gets bigger,
other entries start becoming less than 1 in absolute value and we get
even worse estimates. To deal with this problem, we break $V(\bbr)$ up
into two pieces: the main body, which contains all elements $B\in
V(\bbr)$ with $|b_{11}|\ge1$; and the cusp region, which contains all
elements $B\in V(\bbr)$ with $|b_{11}|<1.$ As $t$ gets bigger, more and
more coefficients of the integral elements of $hG_0R^{(m)}(X)$ will
become 0. Using Proposition \ref{prop:distexplicit}, we know that once
enough entries of $B$ are 0, it will become distinguished and thus
reducible. In Proposition \ref{prop:irredcuspsmall}, we compute the
number of irreducible integral points in the cusp region and in
Proposition \ref{prop:redmainsmall}, we compute the number of
reducible integral points in the main body. They are both negligible
when compared to the number of integral points in the main region and
as a result, we will prove the following result.

\begin{theorem}\label{thmainzcount}
$$N(V(\Z)^{(m)};X) = \frac{1}{\# J^{(m)}[2](\R)}\Vol(\FF\cdot R^{(m)}(X)) + o(X^{\frac{\dim V}{d}}).$$
\end{theorem}
In \S4.5, we show that $\Vol(\FF\cdot R^{(m)}(X))$ grows on the order
of $X^{\frac{\dim V}{d}}$ so the error term is indeed smaller than the
main term.

Let $V(\Z)(b_{11}= 0)$ denote the set of points $B\in V(\Z)$ such
that $b_{11}=0$. Then we have the following proposition:
\begin{proposition}\label{prop:irredcuspsmall}
  With notation as above, we have $N(V(\Z)(b_{11}=0);X)=O_\epsilon(X^{\frac{\dim V-1}{d}+\epsilon})$.
\end{proposition}
\begin{proof}
  It will be convenient to use the following parameters for $T$: $s_i
  = t_i/t_{i+1}$ for $i=1,\ldots,n$; and $s_{n+1}=t_nt_{n+1}.$ The
  condition for $t\in T'$ translates to $s_i>c$ for all $i$. We pick
  the following Haar measure $dh$ on $G(\R)=NTK$:
  \begin{equation}
    \label{eqhaarmeasure}
    \begin{array}{rcl}
dh&=&du\displaystyle\prod_{j=1}^{n-1}s_j^{j(j-2n-1)}\cdot(s_ns_{n+1})^{-\frac{n(n+1)}{2}}d^\times s_jdk\\[0.2in]
&=&du\,\delta(s)d^\times s\,dk,
    \end{array}
  \end{equation}
  where $du$ is a Haar measure on the unipotent group $N$, $dk$ is
  Haar measure on $K$ normalized so that $K$ has volume $1$,
  $\delta(s)$ denotes
  $\prod_{j=1}^{n-1}s_j^{j(j-2n-1)}\cdot(s_ns_{n+1})^{-\frac{n(n+1)}{2}}$,
  and $d^\times s$ denotes $\prod_{j=1}^{n-1}d^\times s_k$.

Then, since $G_0$ is $K$-invariant, \eqref{eqavgtwo} implies that
\begin{equation}\label{eq:firstestimate}
  \begin{array}{rcl}
N(V(\Z)(b_{11}=0);X)&=&O\Bigl(\displaystyle\int_{h\in\FF}\#\{hG_0R^{(m)}(X)\cap V(\Z)(b_{11}=0)\}dh\Bigr)\\[0.2in]
&=&O\Bigl(\displaystyle\int_{u\in N'}\int_{t\in T'}\#\{utG_0R^{(m)}(X)\cap V(\Z)(b_{11}=0)\}\,\delta(s)d^\times s\,du\Bigr)\\[0.2in]
&=&O\Bigl(\displaystyle\int_{t\in T'}\#\{tG_0R^{(m)}(X)\cap V(\Z)(b_{11}=0)\}\,\delta(s)d^\times s\Bigr),
  \end{array}
\end{equation}
where the final equality follows because $N'$ has finite measure,
$utG_0R^{(m)}(X)=t(t^{-1}ut)G_0R^{(m)}(X)$, and the coefficients of
$t^{-1}ut$ are bounded independent of $t\in T'$ and $u\in N'$.

Let $b_{ij},i\le j, (i,j)\neq (n+1,n+2)$ be the system of coordinates
on $V(\R)$, where $b_{ij}$ is the $(i,j)$'th entry of the symmetric
matrix $B$. To each coordinate $b_{ij},$ we associate the weight
$w(b_{ij})$ which records how an element $s\in T$ scales $b_{ij}.$
For example,
\begin{eqnarray*}
  w(b_{11})&=&s_1^{-2}\cdots s_{n-1}^{-2}s_n^{-1}s_{n+1}^{-1}\\
  w(b_{i\,2n+3-i})&=&1,\hspace{96pt}\mbox{coordinates on the anti-diagonal}\\
  w(b_{i\,2n+2-i})&=&s_i^{-1},\quad i=1,\ldots,n,\hspace{10pt}\mbox{coordinates above the anti-diagonal}\\
  w(b_{n+1\,n+1})&=&s_ns_{n+1}^{-1}.
\end{eqnarray*}
Let $C$ be an absolute constant such that $CX^{\frac{1}{d}}$ bounds
the absolute value of all the coordinates of elements $B\in
G_0R^{(m)}(X)$. If, for $(s_1,\ldots,s_{n+1})\in T'$, we have
$CX^{\frac1d}\,w(b_{i_0\,2n+2-i_0})<1$ for some
$i_0\in\{1,\ldots,n+1\}$, then $CX^{\frac1d}\, w(b_{ij})<1$ for all
$i\leq i_0,j\leq 2n+2-i_0.$ Hence the top left $i_0\times(2n+2-i_0)$
block of any integral $B\in t G_0R^{(m)}(X)$ is 0. Just as \cite[Lemma
10.3]{BG3} shows, any such $B$ has zero discriminant. Therefore, to
prove Proposition \ref{prop:irredcuspsmall}, we may assume
\begin{equation}\label{eq:domainfors}
s_i<\frac{X^{1/d}}{C},i=1,\ldots,n;\quad s_{n+1}<\frac{X^{2/d}}{C^2}.
\end{equation}
We use $T_X$ to denote the set of $t=(s_1,\ldots,s_{n+1})\in T'$
satisfying these bounds.

Let $U_1$ denote any subset of the coordinates $b_{ij}$. Let
$V(\R)(U_1)$ denote the subset of $V(\R)$ consisting of elements $B$
whose $(i,j)$ entry is less than 1 in absolute value when $b_{ij}\in
U_1$ and whose $(i,j)$ entry is greater than 1 when $b_{ij}\notin
U_1.$ Let $V(\Z)(U_1)$ denote the set of integral points in
$V(\R)(U_1)$. Then to prove Proposition \ref{prop:irredcuspsmall}, it
suffices to show that
\begin{equation}\label{eq:cuspsmall}
N(V(\Z)(U_1);X)=O_\epsilon(X^{\frac{\dim V-1}{d}+\epsilon}),
\end{equation}
for every set $U_1$ containing $b_{11}$.

Proposition \ref{davlem} in conjunction with the argument used to justify \eqref{eq:firstestimate} implies
\begin{equation*}
  \begin{array}{rcl}
N(V(\Z)(U_1);X)&=&O\Bigl(\displaystyle\int_{t\in T_X}\Vol(tG_0R^{(m)}(X)\cap V(\R)(U_1))\,\delta(s)d^\times s\Bigr)\\[0.2in]
&=&O\Bigl(X^{\frac{\dim V-\#U_1}{d}}\displaystyle\int_{t\in T_X}\prod_{b_{ij}\not\in U_1}w(b_{ij})\,\delta(s)d^\times s\Bigr).
  \end{array}
\end{equation*}
Therefore to prove \eqref{eq:cuspsmall}, we need to estimate:
\begin{equation}\label{eq:cuspsmalltwo}
\w{I}(U_1,X) := X^{\frac{\dim V-\#U_1}{d}}\displaystyle\int_{t\in T_X}\prod_{b_{ij}\not\in U_1}w(b_{ij})\,\delta(s)d^\times s,
\end{equation}
for every set $U_1$ containing $b_{11}$.


Note that if $i'\leq i$ and $j'\leq j$, then $w(b_{i'j'})$ has smaller
exponents in all the $s_k$'s than $w(b_{ij})$. Thus, if a set $U_1$
contains $b_{ij}$ but not $b_{i'j'}$,
then $$\w{I}(U_1\backslash\{b_{ij}\}\cup \{b_{i'j'}\},X) \geq \w{I}(U_1,X).$$
Hence for the purpose of obtaining an upper bound for $\w{I}(U_1,X),$ we may assume that if $b_{ij}\in U_1$,
then $b_{i'j'}\in U_1$ for all $i'\leq i$ and $j'\leq j$.  If such a
set $U_1$ contains any element on, or to the right of, the
off-anti-diagonal, then every element in $V(\Z)(U_1)$ has discriminant
$0$ and by definition $N(V(\Z)(U_1);X)=0$. Let $U_0$ denote the set of coordinates $b_{ij}$ such that $i\leq
j$ and $i+j\leq 2n+1.$ In other words, $U_0$ contains every coordinate to the left of the
off-anti-diagonal. Since every element in $V(\Z)(U_0)$ is distinguished (by Proposition \ref{prop:distexplicit}), hence reducible, it suffices to consider $\w{I}(U_1,X)$ for all $U_1\subsetneq U_0.$

To this end, as the product of the weights over all coordinates is $1$, we define
\begin{equation}
  \label{eqlemmazero}
I(U_1,X)=X^{-\frac{\#U_1}{d}}\int_{s_1,\ldots,s_n=c}^{X^{\frac{1}{d}}}\int_{s_{n+1}=c}^{X^{\frac{2}{d}}}\prod_{(i,j)\in U_1}w(b_{ij})^{-1}\,\prod_{k=1}^{n-1}s_k^{k(k-2n-1)}\cdot(s_ns_{n+1})^{-\frac{n(n+1)}{2}}d^\times s.
\end{equation}
To complete the proof of Proposition \ref{prop:irredcuspsmall}, it suffices to prove the following lemma:
\begin{lemma}\label{lem:combinatoric}
Let $U_1$ be nonempty proper subset of $U_0$. Then
$$I(U_1,X)=O_\epsilon(X^{-\frac{1}{d}+\epsilon}).$$
If $U_1=U_0$ or $U_1=\emptyset$, then $I(U_1,X)=O(1)$.
\end{lemma}
\begin{proof}
  The proof of this lemma is a combinatorial argument using
  induction on $n$.  We first compute
  \begin{equation}\label{eqlemmaone}
I(U_0,X)=X^{-\frac{n(n+1)}{d}}\int_{s_1,\ldots,s_n=c}^{X^{\frac{1}{d}}}\int_{s_{n+1}=c}^{X^{\frac{2}{d}}}s_1s_2^3\cdots s_{n-1}^{2n-3}s_n^{n-1}s_{n+1}^nd^\times s=O(1).
  \end{equation}
This is expected since $V(\Z)(U_0)$ contains all but negligibly many distinguished orbits (see Proposition \ref{prop:redmainsmall}).

Let $U'_1$ denote $U_0\backslash U_1$, and define $I'_n(U'_1,X)$ to equal $I(U_1,X)$. Combining \eqref{eqlemmazero} with \eqref{eqlemmaone}, we obtain
$$I_n'(U'_1,X)=I(U_1,X)=X^{\frac{\#U'_1-n(n+1)}{d}}\int_{s_1,\ldots,s_n=c}^{X^{\frac{1}{d}}}\int_{s_{n+1}=c}^{X^{\frac{2}{d}}}\prod_{(i,j)\in U'_1}w(b_{ij})\cdot s_1s_2^3\cdots s_{n-1}^{2n-3}s_n^{n-1}s_{n+1}^nd^\times s.$$
Write
$U'_1=U'_2\cup U'_3$ where $U'_2$ is the set of coordinates $b_{1j}$
in $U'_1$ and $U'_3=U'_1\backslash U'_2$. Then we may express $I_n'(U'_1,X)$ as the following product:
\begin{equation*}
\Bigl(X^{\frac{\#U'_2-2n}{d}}\int\,\prod_{b_{1j}\in U'_2}w(b_{1j})\,s_1s_2^2\cdots s_{n-1}^2s_ns_{n+1}d^\times s\Bigr)
\Bigl(X^{\frac{\#U'_3-(n-1)n}{d}}\int\,\prod_{b_{ij}\in U'_3}w(b_{ij})\,s_2s_3^3\cdots s_{n-1}^{2n-5}s_n^{n-2}s_{n+1}^{n-1}d^\times s\Bigr).
\end{equation*}
Note that the second term in the above expression is equal to
$I'_{n-1}(\{b_{ij}:b_{i+1,j+1}\in U_3'\},X)$ (which we denote by
$I'_{n-1}(U_3',X)$) and we may estimate it using induction. Denote the
first term in the above expression by $J_n(U_2',X)$. A similar, but
much simpler, induction argument implies
\begin{equation}\label{lem:lem:combin}
  J_n(U'_2,X)=O(X^{-\frac{1}{d}}),
\end{equation}
unless $U'_2=\emptyset,$ in which
  case it is $O(1).$

  Therefore, if $U_2'$ is not empty, then the lemma follows by
  induction on $n$ (used to bound $I'_{n-1}(U_3',X)$ by $O(1)$).  If
  $U_2'$ is empty, then $U_3'$ must be nonempty since $U_1'$ is
  nonempty. If further $U'_3\neq
  U_0\backslash\{b_{11},\ldots,b_{1\,2n}\},$ then by induction, we
  have $I_{n-1}'(U'_3,X)=O_\epsilon(X^{-1/d+\epsilon}).$ The only
  remaining case is when $U_1=\{b_{11},\ldots,b_{1\,2n}\}$, for which
  a direct computation
  yields the result.
\end{proof}

This concludes the proof of Proposition \ref{prop:irredcuspsmall}.
\end{proof}



We now have the following two propositions, whose proofs follow that
of \cite[Lemma 14]{dodpf}.
\begin{proposition}\label{prop:redmainsmall}
  Let $V(\Z)(\phi)^\red$ denote the set of elements in $V(\Z)$ with $b_{11}\neq 0$ that are not irreducible. Then
$$
\int_{G_0}\#\{V(\Z)(\phi)^\red\cap\FF g\cdot R^{(m)}(X)\}dg=o(X^{\frac{\dim V}{d}}).
$$
\end{proposition}

\begin{proposition}\label{prop:bigstabsmall}
  Let $V(\Z)^{{\rm bigstab}}$ denote the set of elements in $V(\Z)$
  which have a nontrivial stabilizer in~$G(\Z)$. Then $$N(V(\Z)^{{\rm
      bigstab}};X)=o(X^{\frac{\dim V}{d}}).$$
\end{proposition}
\begin{proof}
  Observe that if $B\in V(\Z)$ is reducible over $\Z$, then the image
  of $B$ in $V(\F_p)$ is reducible for all $p$. For any prime $p$, let
  $\phi_p$ denote the $p$-adic density of the set of elements of
  $V(\Z_p)$ that are reducible mod $p$. Then to prove
  Proposition \ref{prop:redmainsmall}, it suffices to show $$\prod_p \phi_p = 0.$$
  We show this by proving that $\phi_p$ is bounded above by some
  constant less than 1 when $p$ is large enough. For large enough $p$, there is a positive proportion $r_n$ (depending only on $n$) of polynomials of degree $2n+2$ over $\F_p$ that factors into two linear terms and an irreducible polynomial of degree $2n$. Suppose $f(x)\in \Z_p[x]$ with this reduction type over $\F_p$. Since it has a linear factor, Proposition \ref{prop:poonenover2} implies that there is one distinguished orbit. Since $H^1(\F_p,J)=0$ by Lang's theorem, every orbit is
  soluble. The number of orbits $\#J(\F_p)/2J(\F_p)$ is equal to the
  size of the stabilizer $\#J[2](\F_p).$ Since $f(x)$ has a degree two factor, $\#J[2](\F_p)\geq2.$ Therefore at least $1/2$ of the elements
  in $V_f(\F_p)$ are not distinguished. Hence for $p$ large enough,
$\phi_p \leq 1 - \frac{1}{2}r_n<1.$


We use the same technique to prove Proposition
\ref{prop:bigstabsmall}. For $p$ large enough, there is a positive
proportion $r'_n$ (depending only on $n$) of polynomials of degree
$2n+2$ over $\F_p$ that factors into a linear term and an irreducible
polynomial of degree $2n+1$. If $B\in V_f(\Z_p)$ where $f(x)$ has this
reduction type mod $p$, then $p$ does not divide the discriminant of
$f(x)$. As a consequence, the hyperelliptic curve $y^2=f(x)$ is smooth
over $\rm{Spec}(\Z_p)$ and the 2-torsion of its Jacobian $J[2]$ is a
finite \'{e}tale group scheme over $\rm{Spec}(\Z_p).$ From the
reduction type of $f(x)$ over $p$, we see that
$\#J[2](\Q_p)=\#J[2](\F_p) = 1.$ Denote by $\phi_p$ the $p$-adic
density of the set of elements of $V(\Z_p)$ with non-trivial
stabilizer in $G(\Q_p).$ Then we have shown that $\phi_p \leq 1 - r'_n
< 1$ for $p$ sufficiently large. This completes the proof.
\end{proof}

We may now prove the main result of this section, which we state again for the convenience of the reader.

\begin{theorem}
$$N(V(\Z)^{(m)};X) = \frac{1}{\# J^{(m)}[2](\R)}\Vol(\FF\cdot R^{(m)}(X)) + o(X^{\frac{\dim V}{d}}).$$
\end{theorem}

\begin{proof}
Let $\FF'\subset \FF$ be the set consisting of $h\in\FF$ such that the
$b_{11}$-coefficient of any $B\in hG_0R^{(m)}(X)$ is less than 1 in
absolute value. From \eqref{eqavgtwo}, we see that $N(V(\Z)^{(m)};X)$ is equal to
\begin{equation*}
  \begin{array}{rcl}
&&\displaystyle\frac{1}{\#J^{(m)}[2](\R)\Vol(G_0)}\displaystyle\int_{h\in\FF}\#\{hG_0R^{(m)}(X)\cap V(\Z)^\irr\}dh\\[0.2in]
&=&\displaystyle\frac{1}{\#J^{(m)}[2](\R)\Vol(G_0)}\left(\displaystyle\int_{h\in\FF\backslash\FF'}\#\{hG_0R^{(m)}(X)\cap V(\Z)^\irr\}dh+\displaystyle\int_{h\in\FF'}\#\{hG_0R^{(m)}(X)\cap V(\Z)^\irr\}dh\right).
  \end{array}
\end{equation*}
From Propositions \ref{prop:irredcuspsmall} and
\ref{prop:redmainsmall}, we obtain:
\begin{equation}
  N(V(\Z)^{(m)};X)=\frac{1}{\#J^{(m)}[2](\R)\Vol(G_0)}\displaystyle\int_{h\in\FF\backslash\FF'}\#\{hG_0R^{(m)}(X)\cap V(\Z)\}dh+o(X^{\frac{\dim V}{d}}).
\end{equation}

Note that $b_{11}$ has minimal weight among all the $b_{ij}$. Furthermore,
the length of the projection of $hG_0R^{(m)}(X)$ onto the
$b_{11}$-line is greater than $1$ for $h\in\FF\backslash\FF'$ (by the
definition of $\FF'$). Therefore, for $h\in\FF\backslash\FF'$, all
smaller dimensional projections of $hG_0R^{(m)}(X)$ are bounded by a
constant times its projection onto the $b_{11}=0$
hyperplane. Proposition \ref{davlem} thus implies that
$$
N(V(\Z)^{(m)};X)=\frac{1}{\#J^{(m)}[2](\R)\Vol(G_0)}\displaystyle\int_{h\in\FF\backslash\FF'}\Vol(hG_0R^{(m)}(X))+O\Bigl(\frac{\Vol(hG_0R^{(m)}(X)}{X^{1/d}w(b_{11})}\Bigr)dh+o(X^{\frac{\dim V}{d}}).
$$
Recall $\FF'$ is defined by the condition $CX^{\frac{1}{d}}w(b_{11}) < 1$. Therefore to be in $\FF'$, one of the $s_i$ must be at least $C^{\frac{1}{2n}}X^{\frac{1}{2nd}}.$ Hence the volume of $\FF'$ is bounded by $o(1)$. Moreover, since $\int_{h\in\FF\backslash\FF'}1/w(b_{11})dh=O(1)$, we obtain
\begin{equation}
  \begin{array}{rcl}
    N(V(\Z)^{(m)};X)&=&\displaystyle\frac{1}{\#J^{(m)}[2](\R)\Vol(G_0)}\displaystyle\int_{h\in\FF}\Vol(hG_0R^{(m)}(X))dh+o(X^{\frac{\dim V}{d}})\\[0.2in]
&=&\displaystyle\frac{1}{\#J^{(m)}[2](\R)\Vol(G_0)}\displaystyle\int_{h\in G_0}\Vol(\FF h\cdot R^{(m)}(X))dh+o(X^{\frac{\dim V}{d}})\\[0.2in]
&=&\displaystyle\frac{\Vol(\FF \cdot R^{(m)}(X))}{\#J^{(m)}[2](\R)\Vol(G_0)}\displaystyle\int_{h\in G_0}dh+o(X^{\frac{\dim V}{d}})\\[0.2in]
&=&\displaystyle\frac{\Vol(\FF \cdot R^{(m)}(X))}{\#J^{(m)}[2](\R)}+o(X^{\frac{\dim V}{d}}),
  \end{array}
\end{equation}
where 
the third equality follows because the volume of $\FF h\cdot
R^{(m)}(X)$ is independent of $h$. This concludes the proof of Theorem
\ref{thmainzcount}.
\end{proof}

\subsection{A squarefree sieve}\label{sieve}
In this section, we present versions of Theorem \ref{thmainzcount},
where we count elements (and weighted elements) of $V(\Z)$ satisfying
certain sets of congruence conditions.
\begin{theorem}
  Let $L$ be a subset of $V(\Z)$ defined by finitely many congruence
  conditions on the coefficients of elements in $V(\Z)$. Then
  \begin{equation*}
    N(L\cap V(\Z)^{(m)};X)=N(V(\Z)^{(m)};X)\prod_p\nu_p(L)+o(X^{\frac{\dim V}{d}}),
  \end{equation*}
  where $\nu_p(L)$ denotes the $p$-adic density of $L$ in $V(\Z)$ and
  is equal to $1$ for all but finitely many primes $p$.
\end{theorem}
This theorem follows immediately from the proof of Theorem
\ref{thmainzcount}. (See \cite[Theorem 2.11]{BS} for an analogous situation.)

The following weighted version of Theorem \ref{thmainzcount} also follows immediately:
\begin{theorem}\label{cong3}
  Let $p_1,\ldots,p_k$ be distinct prime numbers. For $j=1,\ldots,k$,
  let $\phi_{p_j}:V(\Z)\to\R$ be a $G(\Z)$-invariant function on
  $V(\Z)$ such that $\phi_{p_j}(B)$ depends only on the congruence
  class of $B$ modulo some power $p_j^{a_j}$ of $p_j$.  Let
  $N_\phi(V^{(m)}(\Z);X)$ denote the number of irreducible
  $G(\Z)$-orbits of $V^{(m)}(\Z)$ having height bounded by $X$, where
  each orbit $G(\Z)\cdot B$ is counted with weight
  $\phi(B)/\#\Stab_{G(\Z)}(B)$; here $\phi$ is defined by
  $\phi(B):=\prod_{j=1}^k\phi_{p_j}(B)$. Then we have
\begin{equation}
N_\phi(V^{(m)}(\Z);X)
  = N(V^{(m)}(\Z);X)
  \prod_{j=1}^k \int_{B\in V(\Z_{p_j})}\tilde{\phi}_{p_j}(B)\,dB+o(X^{\frac{\dim V}{d}}),
\end{equation}
where $\tilde{\phi}_{p_j}$ is the natural extension of ${\phi}_{p_j}$
to $V(\Z_{p_j})$, $dB$ denotes the additive
measure on $V(\Z_{p_j})$ normalized so that $\int_{B\in
  V(\Z_{p_j})}dB=1$, and where the implied constant in the error term
depends only on the local weight functions ${\phi}_{p_j}$.
\end{theorem}

However, in order to prove Theorem \ref{thm:largegood}, we shall need a version of
Theorem \ref{cong3} in which we allow weights to be defined by certain
infinite sets of congruence conditions. The technique for proving such
a result involves using Theorem \ref{cong3} to impose more and more
congruence conditions. While doing so, we need to uniformly bound the
error term. To this end, we have the following proposition proven in
\cite{geosieve}.
\begin{proposition}\label{propunif}
  For each prime $p$, let $W_p$ denote the set of elements $B\in
  V(\Z)$ such that $p^2\mid \Delta(B)$. Then there exists
  $\delta>0$ such that, for any $M>0$, we have
$$
\displaystyle\sum_{p>M}N(W_p;X)=O(X^{\frac{\dim V}{d}}/M^\delta),
$$
where the implied constant is independent of $X$ and $M$.
\end{proposition}

To describe which weight functions on $V(\Z)$ are allowed, we need the
following definition:
\begin{defn}
  A function $\phi:V(\Z)\to[0,1]$ is said to be defined by congruence
  conditions if there exist local functions $\phi_p:V(\Z_p)\to[0,1]$ satisfying the following conditions:
  \begin{enumerate}
  \item For all $B\in V(\Z)$, the product $\prod_p\phi_p(B)$ converges to $\phi(B)$.
  \item For each prime $p$, the function $\phi_p$ is locally constant
    outside some closed set $S_p$ of measure $0$.
  \end{enumerate}
  Such a function is said to be acceptable if, for all sufficiently
  large $p$, we have $\phi_p(B)=1$ whenever $p^2\nmid \Delta(B)$.
\end{defn}

Then we have the following theorem.
\begin{theorem}\label{thinfcong}
  Let $\phi:V(\Z)\to[0,1]$ be an acceptable function that is defined
  by congruence conditions via local functions
  $\phi_p:V(\Z_p)\to[0,1]$. Then, with notation as in Theorem
  \ref{cong3}, we have
$$
N_{\phi}(V^{(m)}(\Z);X)=N(V^{(m)};X)\prod_p\int_{B\in V(\Z_p)}\phi_p(B)dB+o(X^{\frac{\dim V}{d}}).
$$
\end{theorem}
Theorem \ref{thinfcong} follows from Theorems \ref{cong3} and
Proposition \ref{propunif} just as \cite[Theorem 2.21]{BS} followed from
\cite[Theorem 2.12]{BS} and \cite[Theorem 2.13]{BS}.

\subsection{Compatibility of measures and local computations}\label{secvol}
Let $F$ be a large family of hyperelliptic curves. Throughout this section
and the next, we assume without loss of generality that
$\Inv_\infty(F)=S(\R)^{(m)}$ for some fixed integer
$m\in\{0,\ldots,n+1\}$.  To prove Theorem \ref{thm:largegood} we
need to weight each locally soluble element $B\in V(\Z)$ (having
invariants in $\Inv(F)$) by the reciprocal of the number of
$G(\Z)$-orbits in the $G(\Q)$-equivalence class of $B$ in
$V(\Z)$. However, in order for our weight function to be defined by
congruence conditions, we instead define the following weight function
$w:V(\Z)\to [0,1]$:
\begin{equation}\label{eqdefm}
  w(B):=
  \begin{cases}
    \Bigl(\displaystyle\sum_{B'}\frac{\#\Stab_{G(\Q)}(B')}{\#\Stab_{G(\Z)}(B')}\Bigr)^{-1} \qquad&\text{if $B$ is locally soluble and $\Inv(B)\in \Inv(F)$,}\\[.1in]
    \qquad\qquad 0 \qquad&\text{otherwise},
  \end{cases}
\end{equation}
where the sum is over a complete set
of representatives for the action of $G(\Z)$ on the
$G(\Q)$-equivalence class of $B$ in $V(\Z)$.
We then have the following theorem:
\begin{theorem}\label{thmisimp}
Let $F$ be a large family of hyperelliptic curves. Then
\begin{equation}\label{eqtheoremimpone}
\sum_{\substack{C\in F\\H(C)\leq X}}(\#\Sel_2(J(C))-2)=N_w(V(\Z)^{(m)};X)+o(X^{\frac{\dim V}{d}}),
\end{equation}
where $V(\Z)^{(m)}$ is the set of all elements in $V(\Z)$
whose invariants belong to $\Sigma_\infty=S(\R)^{(m)}$.
\end{theorem}
\begin{proof}
  It follows from Proposition \ref{lem:100distinguish} that for a $100\%$ of
  hyperelliptic curves $C(c)\in F$, the set $V_c(\Q)$ has two
  distinguished orbits. Thus, Theorem \ref{theorem:imp} and Corollary
  \ref{cor:imp} show that, up to an error of $o(X^{\frac{\dim V}{d}})$, the left hand side of \eqref{eqtheoremimpone} is equal to
$$
\#(G(\Q)\backslash V_F(\Z)^{\rm ls}_{H<X}),
$$
the number of $G(\Q)$-equivalence classes of elements in $V(\Z)$ that are locally soluble, have invariants in $\Inv(F)$, and have
height bounded by $X$.
Given a locally soluble element $B\in V(\Z)$ such that $\Inv(B)\in F$,
let $B_1\ldots B_k$ denote a complete set of representatives for the action of $G(\Z)$
on the $G(\Q)$-equivalence class of $B$ in $V(\Z)$. Then
\begin{equation}
  \label{eq:qtozimpcomp}
\sum_{i=1}^k \frac{w(B_i)}{\#\Stab_{G(\Z)}(B_i)}=\frac{1}{\#\Stab_{G(\Q)}(B)}\Bigl(\sum_{i=1}^k \frac{1}{\#\Stab_{G(\Z)}(B_i)}\Bigr)^{-1}\sum_{i=1}^k \frac{1}{\#\Stab_{G(\Z)}(B_i)}=\frac{1}{\#\Stab_{G(\Q)}(B)}.
\end{equation}

Therefore, the right hand side of
\eqref{eqtheoremimpone} counts the number of $G(\Q)$-equivalence
classes of elements in $V(\Z)$ that are locally soluble, have
invariants in $F$, and have height bounded by $X$, such that the
$G(\Q)$-orbit of $B$ is weighted with $1/\#\Stab_{G(\Q)}(B)$ for all orbits. The
theorem now follows since $\Stab_{G(\Q)}(B)=1$ for all but negligible
many $B\in V(\Z)$ by Proposition \ref{prop:bigstabsmall}.
\end{proof}

In order to demonstrate that $w$ is defined by congruence conditions, we need to express it
as a local product of weight functions on $V(\Z_p)$. To this end, we define $w_p:V(\Z_p)\to[0,1]$:
\begin{equation}\label{eqmxp}
  w_p(B):=
  \begin{cases}
    \Bigl(\displaystyle\sum_{B'}\frac{\#\Stab_{G(\Q_p)}(B')}{\#\Stab_{G(\Z_p)}(B')}\Bigr)^{-1} \qquad&\text{if $B$ is $\Q_p$-soluble and $\Inv(B)\in \Inv_p(F)$,}\\[.1in]
    \qquad\qquad 0 \qquad&\text{otherwise},
  \end{cases}
\end{equation}
where the sum is over a set of representatives for the action of
$G(\Z_p)$ on the $G(\Q_p)$-equivalence class of $B$ in $V(\Z)$. Our
next aim is to show that $w$ is an acceptable function that is defined
by congruence conditions via the local functions $w_p$.
\begin{proposition}
  If $B\in V(\Z)$ has nonzero discriminant, then
  $w(B)=\prod_pw_p(B)$. Furthermore, $w(b)$ is an acceptable function.
\end{proposition}
\begin{proof}
  The first assertion of the proposition follows from the fact that
  $G$ has class number $1$ over $\Q$.; the proof is identical to that
  of \cite[Proposition 3.6]{BS}. In order to prove that $m$ is acceptable, it
  therefore suffices to check that, for sufficiently large primes $p$,
  we have $w_p(B)=1$ whenever $p^2\nmid\Delta(B)$. This follows from
  Proposition \ref{prop:soluble=integral}.
\end{proof}

From Theorems \ref{thmainzcount} and \ref{thinfcong}, we have the following equality:
\begin{equation}\label{eq:selfirstref}
  N_w(V(\Z)^{(m)};X)=\frac1{\#J^{(m)}[2](\R)}\Vol(\FF\cdot R^{(m)}(X))\prod_p\int_{V(\Z_p)}w(B)dB+o(X^{\frac{\dim V}{d}}).
\end{equation}
For the rest of the section, our aim is to express $\Vol(\FF\cdot
R^{(m)}(X))$ and $\int_{V(\Z_p)}w(B)dB$ in more convenient
forms. To this end, we have the following result that allows us to
compute volumes of multisets in $V(K)$, for $K=\R$ and $\Z_p$. This
result follows from \cite[Proposition 3.11]{BS} and \cite[Proposition 3.12]{BS}.

\begin{proposition}\label{propjacone}
  Let $K$ be $\R$ or $\Z_p$ for some prime $p$, let $|.|$ denote the usual valuation on $K$, and let $s:S(K)\to
  V(K)$ be a continuous section. Then there exists a rational nonzero
  constant $\J$, independent of $K$ and $s$, such that for any
  measurable function $\phi$ on $V(K)$, we have
\begin{equation}
  \begin{array}{rcl}
    \displaystyle\int_{G(K)\cdot
      s(S(K))}\phi(B)d\nu(B)&=&  |\J|\displaystyle\int_{c\in S(K)}\displaystyle\int_{g\in G(K)}
    \phi(g\cdot s(c))d\tau(g) d\mu(c),\\[0.2in]
    \displaystyle\int_{V(K)}\phi(B)d\nu(B)&=&|\J|\displaystyle\int_{\substack{c\in S(K)\\\Delta(c)\neq 0}}\Bigl(\displaystyle\sum_{B\in\textstyle{\frac{V_c(K)}{G(K)}}}\frac{1}{\#\Stab_{G(K)}(B)}\int_{g\in G(K)}\phi(g\cdot B)d\tau(g)\Bigr)d\mu(c).
  \end{array}
\end{equation}
where we regard $G(K)\cdot s(R)$ as a multiset, and
$\frac{V_c(K)}{G(K)}$ denotes a set of representatives for the action
of $G(K)$ on $V_c(K)$.
\end{proposition}

We use Proposition \ref{propjacone} to compute $\Vol(\FF\cdot
R^{(m)}(X))$. If $c\in R^{(m)}$ and $J=J(C(c))$ is the Jacobian of the
corresponding hyperelliptic curve, then the number of $\R$-soluble
$G(\R)$-orbits of $V_c(\R)$ is $\#(J(\R)/2J(\R))$. This number is a
constant independent of $c\in V(\R)^{(m)}$, and we denote it by
$\#(J^{(m)}(\R)/2J^{(m)}(\R))$. Thus, $R^{(m)}$ contains
$\#(J^{(m)}(\R)/2J^{(m)}(\R))$ elements having invariant $c$ for every
$c\in S(\R)^{(m)}$. Therefore, using the first equation of Proposition
\ref{propjacone}, we obtain:
\begin{equation}\label{eq:realvolcomp}
  \begin{array}{rcl}
  \displaystyle\frac1{\#J^{(m)}[2](\R)}\Vol(\FF\cdot R^{(m)}(X))&=&|\J|\,\displaystyle\frac{\#(J^{(m)}(\R)/2J^{(m)}(\R))}{\#J^{(m)}[2](\R)}\Vol(\FF)\Vol(S(\R)^{(m)})\\[0.2in]
&=&|\J|\,a_\infty\Vol(\FF)\Vol(S(\R)^{(m)}),
  \end{array}
\end{equation}
where $a_\nu$ was defined in \eqref{eq:a_nu} for every place $\nu$ of $\Q$.

Next we compute $\int_{V(\Z_p)}w_p(B)d\nu(B)$. Note that since $w_p$ is $G(\Z_p)$-invariant, we have
\begin{equation}\label{eq:wpcomp}
  \begin{array}{rcl}
    \displaystyle\int_{V(\Z_p)}w_p(B)d\nu(B)&=&|\J|_p\Vol(G(\Z_p))\displaystyle\int_{c\in\Inv_p(F)}\Bigl(\displaystyle\sum_{B\in\textstyle{\frac{V_c(\Z_p)}{G(\Z_p)}}}\frac{w_p(B)}{\#\Stab_{G(\Z_p)}(B)}\Bigr)d\mu(c)\\[0.2in]
&=&|\J|_p\,a_p\Vol(G(\Z_p))\Vol(\Inv_p(F)).
  \end{array}
\end{equation}
The final equality follows from a computation similar to
\eqref{eq:qtozimpcomp}; namely, if $J=J(C(c))$ and
$B_c$ is any element in $V_c(\Q_p),$ we have by Proposition
\ref{prop:soluble->integral},
$$\sum_{B\in\frac{V_c(\Z_p)}{G(\Z_p)}}\frac{w_p(B)}{\#\Stab_{G(\Z_p)}(B)}=\frac{\# (G(\Q_p)\backslash V_c^{\rm{sol}}(\Q_p))}{\#\Stab_{G(\Q_p)}(B_c)}=\frac{\#(J(\Q_p)/2J(\Q_p))}{\#J[2](\Q_p)}=a_p.$$

Combining Theorem \ref{thmisimp} with
\eqref{eq:selfirstref}, \eqref{eq:realvolcomp}, and \eqref{eq:wpcomp}, we obtain
\begin{equation}\label{eq:twoselcompfinal}
  \begin{array}{rcl}
    \displaystyle\sum_{\substack{C\in F\\H(C)\leq X}}(\#\Sel_2(J(C))-2)&=&|\J|\,a_\infty\Vol(\FF)\Vol(S(\R)^{(m)})\displaystyle\prod_p|\J|_p\,a_p\Vol(G(\Z_p))\Vol(\Inv_p(F))+o(X^{\frac{\dim V}{\deg H}})\\[0.1in]
&=&\Vol(\FF)\Vol(S(\R)^{(m)})\displaystyle\prod_p\Vol(G(\Z_p))\Vol(\Inv_p(F))+o(X^{\frac{\dim V}{\deg H}}),
  \end{array}
\end{equation}
since $a_{\infty}\prod_pa_p=1$ by \cite[Lemmas 5.7, 5.14]{Stoll01}, and $|\J|\prod_p |\J|_p=1.$

\section{Proof of the main results}
In this section, we prove Theorem \ref{thm:largegood}. Let $F$ be a
large family of hyperelliptic curves. As in the previous section, we
assume without loss of generality that $\Inv_\infty(F)$
is $S(\R)^{(m)}$ for a fixed integer $m\in\{0,\ldots,n+1\}$.

\subsection{The number of hyperelliptic curves in a large family having bounded height}
For any subset $U$ of $S(\Z)$, let $N(U;X)$ denote the number of elements in $U$ having height bounded by $X$.
Our purpose
in this section is to determine asymptotics for $N(\Inv(F);X)$ as $X$ goes
to infinity. To this end, we have the following uniformity estimate proved in \cite{geosieve}.
\begin{proposition}\label{unifhyper}
  For each prime $p$, let $U_p$ denote the set of elements
  $c\in S(\Z)$ such that $p^2\mid\Delta(c)$.  Then there exists
  $\delta>0$ such that, for any $M>0$, we have
$$
\displaystyle\sum_{p>M}N(U_p;X)=O(X^{\frac{\dim V}{d}}/M^\delta),
$$
where the implied constant is independent of $X$ and $M$.
\end{proposition}

Then we have the following theorem which follows from Propositions \ref{davlem} and \ref{unifhyper} just as \cite[Theorem 2.21]{BS} followed from
\cite[Theorem 2.12]{BS} and \cite[Theorem 2.13]{BS}.
\begin{theorem}\label{thheccount}
Let $F$ be a large family of hyperelliptic curves such that $\Inv_\infty(F)=S(\R)^{(m)}$. Then the number of
hyperelliptic curves in $F$ having height bounded by $X$ is $\Vol(S(\R)^{(m)})\prod_p\Vol(\Inv_p(F))$ up to an error of $o(X^{\frac{\dim V}{d}})$.
\end{theorem}

Finally, we also need the following proposition:
\begin{proposition}\label{lem:100distinguish}
  Let $F$ be a large family of hyperelliptic curves. Then for a
  $100\%$ of elements $C\in F$, the class $(\infty')-(\infty)$ is not
  divisible by $2$ in $J(C)(\Q)$.
\end{proposition}
\begin{proof}
By the proof of Theorem \ref{thm:Sentence2}, the element
$(\infty')-(\infty)$ is divisible by $2$ in $J(C)(\Q)$ if and only if
$V_c(\Q)$ has a unique $G(\Q)$-distinguished orbit, where $c$ is the
invariant of $C$. Since $100\%$ of monic degree $2n+2$ integral
polynomials, when ordered by height, correspond to $S_n$-fields, the
result follows from Proposition~\ref{prop:poonenover2}.
\end{proof}

\subsection{The average size of the $2$-Selmer group}
Theorem \ref{thheccount}, Proposition \ref{lem:100distinguish} and
\eqref{eq:twoselcompfinal} imply that
\begin{equation}
  \begin{array}{rcl}
 \lim_{X\to\infty}\displaystyle\frac{\displaystyle\sum_{\substack{C\in
        F\\H(C)<X}}(\#\Sel_2(J(C))-2)}{\displaystyle\sum_{\substack{C\in
        F\\H(C)<X}}1}&=&\displaystyle\frac{\Vol(\FF)\Vol(S(\R)^{(m)})\displaystyle\prod_p\bigl(\Vol(G(\Z_p))\Vol(\Inv_p(F))\bigr)}{\Vol(S(\R)^{(m)})\displaystyle\prod_p\Vol(\Inv_p(F))}\\[.1in] &=&\tau_G,
  \end{array}
\end{equation}
the Tamagawa number of $G$. Since the Tamagawa number of $\PSO$ is $4$
(\cite{Langlands}), Theorem \ref{thm:largegood} follows.

\section{Most monic even hyperelliptic
  curves have only two rational points}
Poonen and Stoll used results from \cite{BG3} and
Chabauty's method to show that a positive proportion of hyperelliptic
curves over $\Q$ having genus $g\geq3$ and a marked rational Weierstrass point have only one
rational point, and that this proportion tends to one as $g$ tends to
infinity (\cite[Theorem 10.6]{PoSto}). In this section, we modify their argument to derive the
analogous result for the family of hyperelliptic curves having a marked rational
non-Weierstrass point, thereby proving Theorem \ref{thchab}.

Before starting the proof of Theorem \ref{thchab}, we sketch the proof
of \cite[Theorem 10.6]{PoSto}.  Given a hyperelliptic curve $C$ with a
marked Weierstrass point $\infty$ and Jacobian $J$, Poonen and Stoll considered the
following diagram (\cite[(6.1)]{PoSto}):
\begin{displaymath}
\xymatrix{
C(\Q)\ar@{^{(}->}[rr] \ar@{^{(}->}[d]
&&C(\Q_2)\ar@{^{(}->}[d]\\
J(\Q)\ar@{^{(}->}[r]\ar@{->>}[d]&\overline{J(\Q)}\ar@{^{(}->}[r]\ar@{->>}[d]&J(\Q_2)\ar@{->>}[r]^\log\ar@{->>}[d]\ar@/^5pc/@{-->}[rrd]^-{\rho\log}&\Z_2^g\ar@{->>}[d]\ar@{.>}[rd]^{\rho}\\
\displaystyle\frac{J(\Q)}{2J(\Q)}\ar@{->>}[r]\ar@{^{(}->}[rd]&\displaystyle\frac{\overline{J(\Q)}}{2\overline{J(\Q)}}\ar[r]
&\displaystyle\frac{J(\Q_2)}{2J(\Q_2)}\ar@{->>}[r]^{\log\otimes\F_2}&\F_2^g\ar@{.>}[r]^-{\P}&\P^{g-1}(\F_2)\\
&\Sel_2(J)\ar[ru] \ar@/_/[rru]_-{\sigma} \ar@(r,d)@{-->}[rrru]_-{\P\sigma}
}
\end{displaymath}

Above, $C(\Q)$ and $C(\Q_2)$ are embedded into $J(\Q)$ and $J(\Q_2)$
via the map $P\mapsto (P)-(\infty)$. The map $\rho$ is defined by
taking the reduction modulo $2$ of the primitive part of $v\in \Z_p^g$
and taking its image under $\P$. Note that the maps $\rho$ and $\P$
are only partially defined, and that the diagram is commutative on
elements for which both maps are defined.  Then the proof of
\cite[Theorem 10.6]{PoSto} follows from these four steps:
\begin{enumerate}
\item The image of $C(\Q_2)$ in $\P^{g-1}(\F_2)$ is usually
  small. More precisely, the average size of $\rho\log(C(\Q_2))$ is
  $6g+9$ (\cite[Corollary 9.10]{PoSto}).
\item If $\sigma$ is injective, then
  $\rho\log(\overline{J(\Q)})\subset \P\sigma(\Sel_2(J))$ (\cite[Lemma
  6.2]{PoSto}).
\item Restrict to a large family of hyperelliptic curves $C$
  such that the image $\rho\log(C(\Q_2))$ is constant in the
  family, say equal to $I$. (That most hyperelliptic curves with genus
  $g$ and a marked rational Weierstrass point are contained in a disjoint union
  of such large families is a consequence of \cite[Lemma
  8.3]{PoSto} and \cite[Propositions 8.5 and 8.7]{PoSto}.) Since
  \cite[Theorem 1.1]{BG3} implies that there are $2$ nontrivial elements of
  $\Sel_2(J)$, on average over $J$, and \cite[Theorem 12.4]{BG3} states that their images in $\F_2^g$ are
  equidistributed, a proportion of at most $\# I2^{1-g}$ curves $C$
  satisfy $$\rho\log(C(\Q_2))\cap \P\sigma \Sel_2(J)\neq \emptyset.$$
  This equidistribution further implies that the proportion of curves for
  which $\sigma$ is not injective is at most $2^{1-g}$.
\item Therefore, aside from a set of density at most $(1+\#I)2^{1-g}$, all curves $C$ in this large family satisfy
$$
C(\Q_2)\cap \overline{J(\Q)}\subset J(\Q_2)_{{\rm tors}},
$$
where $J(\Q_2)_{{\rm tors}}$ denotes the torsion elements in
$J(\Q_2)$. The proof is completed by showing that the density of curves $C$ with $J(\Q)_{{\rm
    tors}}\neq 0$ is zero (\cite[Proposition 8.4]{PoSto}).
\end{enumerate}

We also embed $C(\Q)$ and $C(\Q_2)$ into $J(\Q)$ and $J(\Q_2)$ via the
map $P\mapsto (P)-(\infty)$ and normalize the log map to be surjective
from $J(\Q_2)$ to $\Z_2^g$ as in \cite{PoSto}.  The main difficulty in
adapting their proof in our case is that the image of
$(\infty)-(\infty')$ in $\F_2^g$ does not get equidistributed. Let
$v\in \Z_2^g$ denote the image of $(\infty)-(\infty')$ under the
$\log$ map. Let $v_0$ denote its primitive part and let
$\overline{v_0}$ denote the reduction modulo $2$ of $v_0$. We use
$\langle\cdot\rangle$ to denote the subgroup generated by $\cdot$. We
now consider the following modified version of \cite[(6.1)]{PoSto}:
\begin{displaymath}
\xymatrix{
C(\Q)\ar@{^{(}->}[rr] \ar@{^{(}->}[d]
&&C(\Q_2)\ar@{^{(}->}[d]\\
J(\Q)\ar@{^{(}->}[r]\ar@{->>}[d]&\overline{J(\Q)}\ar@{^{(}->}[r]\ar@{->>}[d]
&J(\Q_2)\ar@{->>}[r]^\log\ar@{->>}[d]\ar@/^5pc/@{-->}[rrrd]^-{\rho'\log}&\Z_2^g\ar@{->>}[r]\ar@{->>}[d]&\Z_2^g/\Z_2\hspace{-2pt}\cdot\hspace{-2pt}v_0\ar@{->>}[d]\ar@{.>}[rd]^{\rho}\\
\displaystyle\frac{J(\Q)}{2J(\Q)}\ar@{->>}[r]\ar@{^{(}->}[rd]&\displaystyle\frac{\overline{J(\Q)}}{2\overline{J(\Q)}}\ar[r]
&\displaystyle\frac{J(\Q_2)}{2J(\Q_2)}\ar@{->>}[r]^{\log\otimes\F_2}&\F_2^g\ar@{->>}[r]&\F_2^g/\langle\overline{v_0}\rangle\ar@{.>}[r]^-{\P}&\P^{g-2}(\F_2)\\
&\Sel_2(J)\ar[ru] \ar@/_/[rru]^-{\sigma} \ar@/_/[rrru]_-{\sigma'}\ar@(r,d)@{-->}[rrrru]_-{\P\sigma'}
}
\end{displaymath}

Our version of Step 1 follows immediately from the proofs of
\cite[Proposition 5.4, Theorem 9.1]{PoSto}, with the only difference
being that in our case the expected size of ${\mathcal C}^{{\rm
    smooth}}(\F_2)$ is bounded above by $4$ instead of $3$, where
$\mathcal{C}$ denotes the minimal proper regular model of $C$. The
reason for this difference is that the $\Z_2$-model of a random
hyperelliptic curve has two smooth points $\infty$ and $\infty'$. The
bound of $2$ on the expected number of other smooth $\F_2$-points
follows from the arguments of \cite[Lemma 9.5]{PoSto}. This yields the
following proposition:
\begin{proposition}
  Let $C$ range over hyperelliptic curves corresponding to elements in
  $\Z_2^{2g+1}\backslash\{\Delta=0\}$ such that $(\infty)-(\infty')\notin J(\Q_2)[2]$. Then $\rho'\log(C(\Q_2))$ is locally constant and its average size is at most $6g+14$.
\end{proposition}

For Step 2, we prove the analogous version of \cite[Lemma 6.2]{PoSto}:
\begin{lemma}\label{lemchabautycovered}
  Suppose $J(\Q)_{{\rm tors}}=0$ and the kernel of $\sigma'$ in
  $\Sel_2(J)$ is equal to the subgroup generated by the class of
  $d_0 = (\infty)-(\infty')$. Then $\rho'\log(\overline{J(\Q)})\subset \P\sigma'
  (\Sel_2(J))$. Furthermore, if $g\in J(\Q)$ has no image under
  $\rho'\log$, then there exist $m$ and $n$ such that $mg=nd_0$.
\end{lemma}
\begin{proof}
  Since $\rho'\log$ is continuous and $\P^{g-2}(\F_2)$ is discrete,
  $\rho'\log(\overline{J(\Q)})=\rho'\log(J(\Q))$.  Since $J(\Q)_{{\rm
      tors}}=0$, we have $J(\Q)/\langle d_0\rangle\equiv
  F\oplus\Z^{r-1}$, where $r$ is the rank of $J(\Q)$ and $F$ is a
  finite abelian group such that any lift $g$ to $J(\Q)$ of an element
  in $F$ satisfies $m g=nd_0$ for some integers $m$ and $n$. This
  implies that such a $g$ has no image under the partially defined map
  $\rho'\log$.

Let $h\in J(\Q)$ be an element that does have an image under
$\rho'\log$. Then the image of $h$ in $F\oplus\Z^{r-1}$ is some
$(t,h')$, where $t\in F$ and $h'\in\Z^{r-1}$. Let $h_0$ denote the
primitive part of $h'$.  Then we have $\rho'\log(h)=\rho'\log(h_0)$
and furthermore, because the kernel of $\sigma'$ is equal to the
subgroup generated by the class of $d_0$, the element $h_0$ has
nonzero image under $\sigma'$. Therefore, we obtain
$\rho'\log(h)=\P\sigma'(h_0)$ which proves the first assertion of the
lemma.

For the second part, let $h\in J(\Q)$ be an element that does not have
an image under $\rho'\log$, and let the image of $h$ in
$F\oplus\Z^{r-1}$ be $(t,h')$, where $t\in F$ and $h'\in\Z^{r-1}$. If
$h'= 0$, then we are done. Otherwise, let $h_0$ denote the primitive
part of $h'$. Since $h$ has no image under $\rho'\log$, neither does
$h_0$, and we have $\log(h_0)\in\Z_2\cdot v_0$. This implies that the
class of $h_0$ in $Sel_2(J)$ maps to $0$ under $\sigma'$ contradicting
our assumption that the kernal of $\sigma'$ is generated by the class
of $d_0$.
%
\end{proof}

For Step 3, we start with the following
analogue of \cite[Theorem 12.4]{BG3}; the proof is identical.
\begin{theorem}\label{thm:equidistribution}
  Fix a place $\nu$ of $\Q$. Let $F$ be a large family of hyperelliptic curves $C$ with a marked non-Weierstrass point such that
  \begin{itemize}
  \item [{\rm (a)}] the cardinality of $J(C)(\Q_\nu)/2J(C)(\Q_\nu)$ is a constant $k$ for all $C\in F$; and
    \item [{\rm (b)}] the set $U_\nu(F)\subset V(\Z_\nu)$, defined to
        be the set of soluble elements in $V(\Z_\nu)$ having
        invariants in $\Inv_\nu(F)$, can be partitioned into $k$ open
        sets $\Omega_i$ such that:
        \begin{itemize}
        \item [{\rm (i)}] for all $i$, if two elements in $\Omega_i$ have the same invariants, then they are $G(\Q_\nu)$-equivalent; and
        \item[{\rm (ii)}] for all $i\neq j$, we have $G(\Q_\nu)\Omega_i\cap G(\Q_\nu)\Omega_j=\emptyset$.
        \end{itemize}
  \end{itemize}
  $($In particular, the groups $J(C)(\Q_\nu)/2J(C)(\Q_\nu)$ are
  naturally identified for all $C\in F$.$)$ Then when elements $C\in
  F$ are ordered by height, the images of the non-distinguished
  elements $($i.e., elements that do not corresponed to either the
  identity or the class of $(\infty')-(\infty)$ in $J(C)(\Q))$ under
  the map
$$
\Sel_2(J(C))\to J(C)(\Q_\nu)/2J(C)(\Q_\nu)
$$
are equidistributed.
\end{theorem}

Let $F$ be a large family of hyperelliptic curves corresponding to an
open subset of $\Z_2^{2g+1}\backslash\{\Delta=0\}$ such that $F$
satisfies the hypothesis of Theorem \ref{thm:equidistribution} and the
image of $\rho'\log(C(\Q_2))$ in $\P^{g-2}(\F_2)$ is constant for
$C\in F$. We denote this image by $I$. We may further assume that the
log maps are normalized such that the image $v$ of
$(\infty)-(\infty')$ is constant throughout this family
(cf. \cite[Proposition 8.2]{PoSto}).

On average over the Jacobians $J$ of the curves in $F$, there are $4$
non-distinguished elements in $\Sel_2(J)$, and the images of these
elements under $\sigma$ are equidistributed in $\F_2^g$.
Therefore, a
proportion of at least $1-\#I2^{3-g}$ curves $C$ in $F$ satisfy
$$
\rho'\log(C(\Q_2))\cap\P\sigma'(\Sel_2(J))=\emptyset.
$$ Furthermore, a proportion of at most $2^{3-g}$ curves fail to
satisfy the conditions of Lemma \ref{lemchabautycovered} (we need the
image of $\sigma$ to avoid both $0$ and $\overline{v_0}$).  Say that a
point $P\in C(\Q)\backslash\{\infty,\infty'\}$ is {\it bad} if there
exist integers $m$ and $n$, not both zero, such that
$$m((P)-(\infty)) = n((\infty)-(\infty')).$$
Therefore,
aside from a set of density at most $(1+\#I)2^{3-g}$, all curves $C\in F$ are such that every point $P\in C(\Q)\backslash\{\infty,\infty'\}$ is bad.



We summarize the above discussion in the following theorem.

\begin{theorem}\label{thm:chabauty}
Suppose $C$ is an even degree hyperelliptic curve of genus $g$ over $\Q$ satisfying the following three conditions:
\begin{enumerate}
\item[{\rm (1)}] $J(\Q)_{\rm tors}=0$,
\item[{\rm (2)}] $\ker\sigma' = \langle (\infty)-(\infty') \rangle$,
\item[{\rm (3)}] $\rho'\log(C(\Q_2))\cap \P\sigma'(\Sel_2(J)) = \emptyset$.
\end{enumerate}
Then every point $P\in C(\Q)\backslash\{\infty,\infty'\}$ is bad, i.e, there exist integers $m$ and
 $n$, not both $0$, such that $$m((P)-(\infty)) = n((\infty)-(\infty')).$$

Moreover, the proportion of even degree hyperelliptic curves
$C$ of genus $g$ over $\Q$ satisfying the above three conditions is at
least $1 - (48g + 120)2^{-g}.$
\end{theorem}

We say that a monic even degree hyperelliptic curve $C$ over $\Q$ is good if $C(\Q)$ has no bad points.
Then we have the following theorem:
\begin{theorem}\label{thm:hilbert}
A proportion of $100\%$ of monic even degree hyperelliptic curves over $\Q$ having fixed genus $g\geq 4$ are good.
\end{theorem}

We work $p$-adically for some fixed prime $p$. Suppose $C$ is an monic even degree
hyperelliptic curve with coefficients in $\Z_p$. Let $\ell:C(\Q_p)\rightarrow \Z_p^g$
denote the map sending $P\in C(\Q_p)$ to $\log((P)-(P^\tau))$ where
$\tau$ denotes the hyperelliptic involution and $\log$ is computed
with respect to the differentials
$$\{dx/y,xdx/y,\ldots,x^{g-1}dx/y\}.$$
We say a point $P\in
C(\Q_p)\backslash\{\infty,\infty'\}$ is \emph{bad} if the
$\Z_p$-lines spanned by $\ell(P)$ and $\ell(\infty)$ have nonzero
intesections.


We thank Jacob Tsimerman for several conversations which led to the
proof of the following theorem, from which Theorem \ref{thm:hilbert}
follows immediately.
\begin{theorem}\label{thm:hilbertp}
Suppose $g\geq4$. The set $U$ of elements in $\Z_p^{2g+1}\backslash\{\Delta=0\}$ corresponding to hyperelliptic curves $C$ of genus $g$ such that $C(\Q_p)\backslash\{\infty,\infty'\}$ contains no bad points is dense. Furthermore, the $p$-adic closure of its complement has measure $0$.
\end{theorem}
\begin{proof}
An element $v\in\Z_p^{2g+1}\backslash\{\Delta=0\}$ yields a
hyperelliptic curve $C$ along with a point $\infty$. Let $P\in
C(\Q_p)$ be a point such that $P\neq \infty$. The pair $(C,P)$ then
corresponds to an element $v'\in\Z_p^{2g+1}\backslash\{\Delta=0\}$
such that $v'\neq v$. Furthermore, as $P\to \infty$ in $C(\Q_p)$, we
have $v'\to v$.  We say that a pair of points $(P,Q)\in C(\Q_p)\times
C(\Q_p)$ is a {\it bad pair} if $P\neq Q$, $P\neq Q^\tau$, and the
$\Z_p$-lines spanned by $\ell(P)$ and $\ell(Q)$ have a nonzero
intersection. We show in Lemma \ref{badpair} that the number of bad
pairs $(P,Q)\in C(\Q_p)\times C(\Q_p)$ is finite for any monic even
degree hyperelliptic curve over $\Q_p$. From this it follows that
given a pair $(C,\infty)$ corresponding to
$v\in\Z_p^{2g+1}\backslash\{\Delta=0\}$, there exist points $P$
arbitrarily close to $\infty$ such that $P$ is not part of any bad
pair. It thus follows that there exist points
$v'\in\Z_p^{2g+1}\backslash\{\Delta=0\}$ (corresponding to such pairs
$(C,P)$), arbitrarily close to $v$, that correspond to hyperelliptic
curves containing no bad points. Therefore, $U$ is dense.

Let $V$ denote the complement of $U$ in
$\Z_p^{2g+1}\backslash\{\Delta=0\}$. We claim that $V$ is a $p$-adic
subanalytic subset of $M$. The theory of subanalytic sets is studied in great detail in \cite{DD}. We do not repeat the definition of subanalytic sets and instead remark that subanalytic sets are stable under projections onto coordinate hyperplanes and that sets defined by the vanishing and nonvanishing of analytic functions are subanalytic. Moreover, being subanalytic is a ($p$-adic) local property. The \emph{dimension} of a subanalytic set is defined to be the maximal dimension of a $p$-adic manifold contained in it (\cite[3.15]{DD}). This notion of dimension behaves as one expected: a $0$-dimensional subanalytic set is finite; the dimension of the boundary $\bar{A}\backslash A$ of a subanalytic set $A$ is less than the dimension of $A$ (\cite[3.26]{DD}).  

We now show that $V$ is a $p$-adic subanalytic subset of $M$. It suffices to
check this locally. Restrict to an open subset $W$ of
$\Z_p^{2g+1}\backslash\{\Delta=0\}$ such that $\mathcal{C}^{\rm smooth}(\F_p)$
is constant (having size $k$) for curves $C$ corresponding to
elements in $W$ where $\mathcal{C}$ denote the minimal proper regular model of $C$. Then the moduli space of pairs $(C,P)$, where $C$ is
a curve corresponding to an element in $W$ and $P$ is a point in
$C(\Q_p)$, is isomorphic to $W\times \mathcal{C}^{\rm smooth}(\F_p)\times
\Z_p$. The set of pairs $(C,P)$ corresponding to elements in this
moduli space such that $P$ is a bad point of $C(\Q_p)$ is a subanalytic set of $W\times \mathcal{C}^{\rm smooth}(\F_p)\times
\Z_p$ defined by $\ell(P),\ell(\infty)\neq0$ and $\ell(P)/\!/\ell(\infty)$. Since subanalytic sets are preserved by projections, this
implies that $V\cap W$ is subanalytic in $W$, as necessary. We have
already proven that $V$ does not contain any $p$-adic open ball of dimension $2g+1$ as its complement is dense. Hence
its dimension as a subanalytic set (\cite[3.15]{DD}) is less than
$\dim(\Z_p^{2g+1}\backslash\{\Delta=0\})=2g+1$.  Moreover, the
dimension of $\bar{V}\backslash V$ is less than the dimension of $V$
(\cite[3.26]{DD}), where $\bar{V}$ denotes the $p$-adic closure of
$V$. Therefore, the $p$-adic closure of $V$ has
measure $0$ as necessary.
\end{proof}

We now have the following lemma which was assumed in the proof of
Theorem \ref{thm:hilbertp}.
\begin{lemma}\label{badpair}
Let $C$ be a monic even degree hyperelliptic curve with coefficients
in $\Z_p$, having genus $g\geq 4$. Then the set of bad pairs $(P,Q)\in
C(\Q_p)\times C(\Q_p)$ is finite.
\end{lemma}
\begin{proof}
Let $\Sigma$ denote the subset of $C(\Q_p)\times C(\Q_p)$ consisting of bad pairs $(P,Q)$. Then $\Sigma$ is subanalytic as it is defined by $\ell(P),\ell(Q)\neq0$ and $\ell(P)/\!/\ell(Q)$. We will show that the dimension of $\Sigma$ as a subanalytic set is zero which implies that $\Sigma$ is finite by \cite[3.26]{DD}.

Let $P\in C(\Q_p)$ be any
point. Restricting $\ell$ to a neighborhood $W_P$ around $P$ gives
an analytic function $\ell_P:\Z_p\to\Z_p^g$. The main difficulty in
proving this lemma is that it is difficult to explicitly compute the
function $\ell_P$. However, for any $P'$ in the residue disk around
$P$, $\ell_P(P')$ is the sum of $\ell_P(P)$ and twice a $p$-adic
integral. Hence we can compute the derivative of $\ell_P$ using the
fundamental theorem of calculus and obtain:
\begin{equation}\label{eqder}
\begin{array}{rclcl}
\ell_P'(P') &=&
\big(\displaystyle\frac{2}{y(P')},\displaystyle\frac{2x(P')}{y(P')},\ldots,\displaystyle\frac{2x(P')^{g-1}}{y(P')}\big),\quad&&\mbox{if
}P'\not\in\{\infty,\infty'\},\\[.2 in] \ell_P'(P') &=& (0,0,\ldots,\pm2),\quad&&\mbox{if
}P'\in\{\infty,\infty'\}.
\end{array}
\end{equation}
The second formula follows from applying a change of variable
$t=1/x,s=y/x^{g+1}$ and then using the fundamental theorem of
calculus. One key fact to notice is that the projections of
$\ell_P'(P')$ and $\ell_Q'(Q')$ onto any $2$-dimensional coordinate
hyperplane corresponding to two consecutive coordinates are $\Q_p$-parallel if and only if $P'=Q'$ or $P' =
Q'^\tau.$ This observation yields
the following lemma:

\begin{lemma}\label{lemmasimpler}
For a fixed point $P\in
C(\Q_p)$, the set of points $Q\in C(\Q_p)$ such that $(P,Q)$ is a bad
pair is finite.
\end{lemma}
\begin{proof}
Indeed, the intersection of $\Q_p\cdot \ell(P)$ and
$\ell(C(\Z_p))$ is a subanalytic set of dimension at most $1$. Hence
it is either finite or contains an open ball $B$. If it is finite,
then we are done. Otherwise, the derivatives $\ell'(Q)$ are all
parallel to $\ell(P)$ for $Q\in B$, which is a
contradiction.
\end{proof}

Let $(P,Q)\in C(\Q_p)\times C(\Q_p)$ be a bad pair. Since $\ell(P)$ and $\ell(Q)$ are $\Q_p$-parallel, there exists a
coordinate, say $j$, for which both $\ell(P)$ and $\ell(Q)$ are
nonzero and have the smallest $p$-adic valuation among all nonzero
coordinates. Hence there exist small neighborhoods $W_P$ and $W_Q$ of $P$
and $Q$, respectively, such that the $j$-th coordinates of
$\ell_P(P')$ and $\ell_Q(Q')$ are nonzero and have the smallest
$p$-adic valuation among all nonzero coordinates for any $P'\in W_P$,
and $Q'\in W_Q$. Moreover since $P\neq Q$ and $P\neq Q^\tau$, we may
further assume that $P'\neq Q'$ and $P'\neq Q'^\tau$ for any $(P',Q')\in
W_P\times W_Q$.  For any $i=1,\ldots,g$ and any vector $v\in \Q_p^g,$
we write $v_i$ for the $i$-th coordinate of $v$ and write $v^{(i)}$
for the vector in $\Q_p^{g-1}$ obtained from $v$ by removing the
$i$-th coordinate. For any $P'\in W_P$, write $f_P(P')\in\Q_p^{g-1}$
for the vector
$$f_P(P') =
\big(\frac{\ell_P(P')_1}{\ell_P(P')_j},\ldots,\frac{\ell_P(P')_g}{\ell_P(P')_j}\big)^{(j)}.$$
Similarly define $f_Q(Q')$ for $Q'\in W_Q$. Then $(P',Q')$ is a bad
pair if and only if $f_P(P')=f_Q(Q')$. Let $h:W_P\times W_Q\rightarrow
\Z_p^{g-1}$ denote the analytic function $h(P',Q')=f_P(P')-f_Q(Q')$
and let $S$ denote the vanishing locus of $h$. Then $S$ is an analytic
subset of $W_P\times W_Q$ and its projections to $W_P$ and $W_Q$ are
subanalytic. Computing the partial
derivatives of $h$ at $(P,Q)$ gives
$$h_P(P,Q) = \frac{1}{\ell_P(P)_j}\ell_P'(P)^{(j)} -
\frac{\ell_P'(P)_j}{\ell_P(P)^2_j}\ell_P(P)^{(j)},\quad h_Q(P,Q) = -
\frac{1}{\ell_Q(Q)_j}\ell_Q'(Q)^{(j)} +
\frac{\ell_Q'(Q)_j}{\ell_Q(Q)^2_j}\ell_Q(Q)^{(j)}.$$ Hence if both of these
partial derivatives are zero, then the vectors $\ell_P'(P)^{(j)}$,
$\ell_P(P)^{(j)}$, $\ell_Q(Q)^{(j)}$, $\ell_Q'(Q)^{(j)}$ are all
$\Q_p$-parallel which leads to a contradiction since $g\geq4$ and
$P\neq Q,Q^\tau$. Note $\ell_P(P)^{(j)}$, $\ell_Q(Q)^{(j)}$ are parallel
because $\ell_P(P)$ and $\ell_Q(Q)$ are parallel. We assume without loss of
generality that $h_P(P,Q)\neq 0$.

Let $S_P$ denote the image of $S$ under the projection map from
$W_P\times W_Q$ to $W_P$. If the dimension of $S_P$ as a subanalytic
set is $0$, then it is a finite set and Lemma \ref{badpair} follows from Lemma \ref{lemmasimpler}.
If the dimension of $S_P$ as a
subanalytic set is $1$, then it contains an open
ball. Replacing $W_P$ by this open ball, we can assume that
$S_P=W_P$. Since $h_P(P,Q)\neq 0$, the implicit function theorem
implies that there exists an analytic section $W_P\rightarrow S$ and
composing it with the second projection gives an analytic map
$s:W_P\rightarrow W_Q$ such that $(P',s(P'))$ is a bad pair for any
$P'\in W_P$. Let $\alpha:W_P\rightarrow \Q_p^\times$ denote the
analytic function such that
\begin{equation}\label{eq:d0}
\ell_Q(s(P'))=\alpha(P')\ell_P(P'),
\end{equation}
for any $P'\in W_P$. The vanishing set of the derivative $s'$ of $s$
is analytic and hence is either finite or contains an open ball. In
the latter case, $s$ is contant on this open ball which contradicts
Lemma \ref{lemmasimpler}. By replacing $W_P$ by an open ball inside it, we may assume
that $s'(P')\neq0$ for any $P'\in W_P$. Note that $W_P$ might not
contain $P$ anymore. Differentiating \eqref{eq:d0} gives
\begin{equation}\label{eq:d1}
\ell_Q'(s(P')) = \alpha_1(P')\ell_P(P') + \alpha_2(P')\ell_P'(P'),
\end{equation}
with $\alpha_1=\alpha'/s'$ and $\alpha_2=\alpha/s'$. Differentiating
\eqref{eq:d1} again shows that the vectors $\ell_Q''(s(P'))$, $\ell_P(P')$,
$\ell_P''(P')$, $\ell_P'(P')$ are linearly dependent over $\Q_p$ for any
$P'\in W_P$. Since $P'\neq Q'$ and $P'\neq Q'^\tau$ for any $P'\in
W_P$ and $Q'\in W_Q$ by assumption, we see that $\ell_Q'(s(P'))$ and
$\ell_P'(P')$ are not parallel and hence $\ell_P(P')$ can be written as a
linear combination of $\ell_Q'(s(P'))$ and $\ell_P'(P')$ by
\eqref{eq:d1}. Therefore, the vectors $\ell_Q''(s(P'))$, $\ell_Q'(s(P'))$,
$\ell_P''(P')$, $\ell_P'(P')$ are linearly dependent for any $P'\in W_P$.

Shrink $W_P$ if necessary so that $W_P$ does not contain $\infty$
or $\infty'$. This allows us to have a uniform formula for the
derivative of $\ell$. An elementary determinant computation (using the
first $4$ coordinates, which requires $g\geq4$) shows that the vectors
$\ell_Q''(Q')$, $\ell_Q'(Q')$, $\ell_P''(P')$, $\ell_P'(P')$ are linearly
dependent if and only if $P'=Q'$ or $P'=Q'^\tau$ neither of which is
true if $Q'=s(P')$ and $P'\in W_P$. This completes the proof of
Lemma \ref{badpair}.
\end{proof}

Theorem \ref{thchab} follows from Theorem \ref{thm:chabauty} and Theorem \ref{thm:hilbert}.

\subsection*{Acknowledgments}

We are very grateful to Manjul Bhargava and Benedict Gross for
suggesting this problem to us and for many helpful conversations. We
are also very grateful to Bjorn Poonen for explaining Chabauty's
method to us and for helpful comments on earlier versions of the
argument. We are extremely grateful to Cheng-Chiang Tsai, Jacob
Tsimerman, and Ila Varma for several helpful conversations. The first
author is grateful for support from NSF grant DMS-1128155. The second
author is grateful for support from a Simons Investigator Grant and
NSF grant~DMS-1001828.

\end{document}